\documentclass[12pt,leqno]{amsart}
\usepackage{amssymb}
\usepackage{euscript}
\usepackage[hypertex]{hyperref}

\def\today{July 2012}
\catcode`\@=11
\def\@evenfoot{\rule{0pt}{20pt}[\today] \hfill}
\def\@oddfoot{\rule{0pt}{20pt}\hfill [\today]}

\title[Crossed intervals and the Hochschild]%
       {Crossed interval groups and operations on the Hochschild cohomology}

\author[M.~Batanin and M.~Markl]{Michael Batanin and Martin~Markl}
\thanks{The first author gratefully acknowledges for the financial
support of Scott Russel Johnson Memorial Foundation, Max Plank
Institut f\"{u}r Mathematik and Australian Research Council (grant
No.~DP0558372). The second author was supported by the grant GA \v CR
201/08/0397 and RVO: 67985840.}

\address{Macquarie University, North Ryde, 2109, NSW, Australia}
\email{mbatanin@ics.mq.edu.au}
\address{Mathematical Institute of the Academy, {\v Z}itn{\'a} 25,
         115 67 Prague 1, The Czech Republic}
\email{markl@math.cas.cz}


\newtheorem{theorem}{Theorem}[section]

\newtheorem{lemma}[theorem]{Lemma}
\newtheorem{proposition}[theorem]{Proposition}

\newtheorem*{theoremA}{Theorem~A}
\newtheorem*{theoremA'}{Theorem~A'}
\newtheorem*{theoremB}{Theorem~B}
\newtheorem*{theoremC}{Theorem~C}

\theoremstyle{definition}
\newtheorem{example}[theorem]{Example}
\newtheorem{remark}[theorem]{Remark}
\newtheorem{variants}[theorem]{Variants}
\newtheorem{definition}[theorem]{Definition}
\newtheorem*{warning}{Warning}
\newtheorem*{notation}{Notation}

\DeclareMathOperator{\Hom}{Hom}
\DeclareMathOperator{\Lin}{Lin}
\DeclareMathOperator{\Aut}{Aut}

\begin{document}
\baselineskip18pt plus 1pt minus 1pt

\bibliographystyle{plain}

 \def\Latnic{{\EuScript L}} 
\def\Lat#1{{\EuScript L_{(#1)}}}
\def\white{\thicklines
{
\unitlength=.700000pt
\begin{picture}(50.00,35.00)(-10.00,12.00)
\put(25.00,11.00){\makebox(0.00,0.00){\scriptsize $\cdots$}}
\put(20.00,30.00){\makebox(0.00,0.00){$\circ$}}
\put(40.00,10.00){\line(-1,1){18.00}}
\put(10.00,10.00){\line(1,2){9.00}}
\put(0.00,10.00){\line(1,1){18.00}}
\end{picture}}
}

\def\rwhite{\thicklines
{
\unitlength=.700000pt
\begin{picture}(80.00,35.00)(-35.00,-15)
\put(0.00,0.00){\makebox(0.00,0.00){$\bullet$}}
\put(-23.00,-15.00){\makebox(0.00,0.00)[rb]{\scriptsize $n\!\!-\!\!1$}}
\put(-20.5,-20.50){\makebox(0.00,0.00){\white}}
\put(0.00,0.00){\line(1,-1){36.00}}
\put(0.00,0.00){\line(0,1){19.00}}
\put(0.00,0.00){\line(-1,-1){17.00}}
\end{picture}}
}

\def\lwhite{\thicklines
{
\unitlength=.700000pt
\begin{picture}(80.00,35.00)(-35.00,-15)
\put(23.00,-15.00){\makebox(0.00,0.00)[lb]{\scriptsize $n\!\!-\!\!1$}}
\put(0.00,0.00){\makebox(0.00,0.00){$\bullet$}}
\put(18.00,-20.00){\makebox(0.00,0.00){\white}}
\put(0.00,0.00){\line(-1,-1){36.00}}
\put(0.00,0.00){\line(0,1){19.00}}
\put(0.00,0.00){\line(1,-1){17.00}}
\end{picture}}
}

\def\cwhite{\thicklines
{
\unitlength=.700000pt
\begin{picture}(100.00,35.00)(-45.00,-15)
\put(0.00,4.00){\makebox(0.00,0.00){$\circ$}}
\put(8,9.00){\makebox(0.00,0.00)[l]{\scriptsize $n\!\!-\!\!1$}}
\put(6,-11){\makebox(0.00,0.00){\scriptsize $i$}}
\put(-2.00,2.00){\line(-1,-1){38.00}}
\put(-1.00,2.00){\line(-3,-4){28.5}}
\put(2.00,2.00){\line(1,-1){38.00}}
\put(0.00,8.00){\line(0,1){15.00}}
\put(0.00,1.00){\line(0,-1){19.00}}
\put(0.00,-19.00){\makebox(0.00,0.00){$\bullet$}}
\put(0.00,-19.00){\line(1,-2){8.50}}
\put(0.00,-19.00){\line(-1,-2){8.50}}
\put(25.50,-39.00){\makebox(0.00,0.00)[b]{\scriptsize $\cdots$}}
\put(-17.50,-39.00){\makebox(0.00,0.00)[b]{\scriptsize $\cdots$}}
\end{picture}}
}

\def\cwhitestub{\thicklines
{
\unitlength=.6pt
\begin{picture}(100.00,35.00)(-45.00,-15)
\put(0.00,4.00){\makebox(0.00,0.00){$\circ$}}
\put(8,9.00){\makebox(0.00,0.00)[l]{\scriptsize $n\!\!+\!\!1$}}
\put(6,-12){\makebox(0.00,0.00){\scriptsize $i$}}
\put(-2.00,2.00){\line(-1,-1){38.00}}
\put(-1.00,2.00){\line(-3,-4){28.5}}
\put(2.00,2.00){\line(1,-1){38.00}}
\put(0.00,8.00){\line(0,1){15.00}}
\put(0.00,1.00){\line(0,-1){35.00}}
\put(0.00,-35){\makebox(0.00,0.00){$\bullet$}}
\put(22.50,-39.00){\makebox(0.00,0.00)[b]{\scriptsize $\cdots$}}
\put(-13.50,-39.00){\makebox(0.00,0.00)[b]{\scriptsize $\cdots$}}
\end{picture}}
}

\def\calU{{\mathcal U}}
\def\bbN{{\mathbb N}} \def\epi{\twoheadrightarrow}
\def\Frmu{\Fr}\def\Norm{{\rm Nor\/}} \def\NormB{\Norm(\Big)}
\def\NormT{\Norm(\Tam)}
\def\End{{{\mathcal E}\hskip -.1em {\it nd}}}
\def\END{{\rm End}}\def\uV{{\underline V}}\def\omu{{(\O,\mu)}}
\def\fmu{{\Fr}}\def\pom{{\rada{u_1o_1}{u_no_n}}}
\def\bs{\hbox{$\EuScript N$}}
\def\F{{\EuScript F}} \def\D{{\mathcal D}}\def\uD{{\underline {\D}}}
\def\Tot{{\rm Tot\/}} \def\SB{{\EuScript B}} \def\Span{{\it Span}}
\def\U{{\EuScript U}} \def\uTot{\overline{\rm Tot\/}} 
\def\UAss{\hbox{$U \hskip -.2em  \it{\mathcal A}ss$}}
\def\sUAss{\hbox{\scriptsize $U \hskip -.2em  \it{\mathcal A}ss$}}
\def\sAss{\hbox{\scriptsize$\it{\mathcal A}ss$}}
\def\lTot{\underline{\rm Tot\/}} \def\bara{{\overline a}}
\def\C{{\EuScript C}} \def\bbZ{{\mathbb Z}}
\def\bbT{{\mathbb T}}
\def\H{{\EuScript H}} \def\Dual{{\it Dual}} \def\calP{{\mathcal P}}
\def\Br{{\EuScript Br}} \def\barg{{\overline g}}
\def\Z{{\EuScript Z}} \def\MO{{\it MO}} \def\SYMCAT{{\rm SymCat}_\bullet}
\def\Flip{{\EuScript Flip}} \def\Nat{{\EuScript Nat}}\def\Symcat{\SYMCAT}
\def\NAT{{\check B}}\def\B{{B}}
\def\MultSO{{\rm MultSO}}\def\SymCat{\SYMCAT}\def\SO{{\rm SO}}
\def\S{{\EuScript S}}\def\O{{\EuScript O}}\def\P{{\EuScript P}}
\def\SC{{\mathit C}}\def\Fr{\EuScript Fr}
\def\Dis{{\EuScript D}}\def\mezi#1#2{{1 \leq #1 \leq #2}}
\def\pa{\partial}\def\id{{\mathrm {{id}}}}
\def\I{\EuScript I}\def\int#1{{\langle #1 \rangle}}
\def\rada#1#2{{#1,\ldots,#2}} \def\op{{\rm op}}
\def\bar#1{{B_{#1}^{\mbox{\scriptsize $A$-$A$}}(A)}}
\def\ot{\otimes} \def\otexp#1#2{{#1}^{\otimes #2}}
\def\Hoch{{\rm Hoch}}  \def\Big{{\mathcal B}}
\def\open#1{\stackrel{\raisebox{-.2em}{\scriptsize\int {k-1},\int{l-1} $\mathrm o$}}{#1}}
\def\AbiMod{\mbox{$A$-{\tt biMod}}} \def\joy{\mathit{joy}}
\def\IS{{\I S}} \def\kMod{\mbox{$\mathbf k$-{\tt Mod}}}
\def\Rada#1#2#3{#1_{#2},\dots,#1_{#3}}
\def\CH#1{C^{#1}(A;A)}
\def\Tam{{\mathcal T}} \def\Br{{\mathcal B}r} \def\KS{{\mathcal M}}
\def\BN{{\widehat \Big}}\def\TN{{\widehat \Tam}}
\def\NBr{\widehat{\Br}}
\def\exeptional{$\ \unitlength 1em \thicklines\line(0,1){1}\ $}
\def\stub{\ 
\unitlength 1em 
\begin{picture}(0,1)
\thicklines
\put(0,0){\line(0,1){1}}
\put(0,0){\makebox(0,0){$\bullet$}}
\end{picture}
}
\def\uChain{{\underline{\Chain}}}
\def\omu{\O}\def\pnu{\P}

\def\cases#1#2#3#4{
                  \left\{
                         \begin{array}{ll}
                           #1,\ &\mbox{#2}
                           \\
                           #3,\ &\mbox{#4}
                          \end{array}
                   \right.
}

\newcommand{\Square}[8]{
\setlength{\unitlength}{.8cm}
\begin{picture}(5,3.6)
\thicklines

\put(0,3){\makebox(0,0){$#1$}}
\put(5,3){\makebox(0,0){$#2$}}
\put(0,0){\makebox(0,0){$#3$}}
\put(5,0){\makebox(0,0){$#4$}}

\put(-.5,1.5){\makebox(0,0)[r]{$#6$}}
\put(5.5,1.5){\makebox(0,0)[l]{$#7$}}
\put(2.5,0.5){\makebox(0,0)[c]{$#8$}}
\put(2.5,3.5){\makebox(0,0)[c]{$#5$}}

\put(1,0){\vector(1,0){3}}
\put(1,3){\vector(1,0){3}}
\put(0,2.5){\vector(0,-1){2}}
\put(5,2.5){\vector(0,-1){2}}
\end{picture}
}

\newcommand{\SQuare}[8]{
\setlength{\unitlength}{.8cm}
\begin{picture}(5,3.6)
\thicklines

\put(0,3){\makebox(0,0){$#1$}}
\put(5,3){\makebox(0,0){$#2$}}
\put(0,0){\makebox(0,0){$#3$}}
\put(5,0){\makebox(0,0){$#4$}}

\put(-.5,1.5){\makebox(0,0)[r]{$#6$}}
\put(5.5,1.5){\makebox(0,0)[l]{$#7$}}
\put(2.5,0.5){\makebox(0,0)[c]{$#8$}}
\put(2.5,3.5){\makebox(0,0)[c]{$#5$}}

\put(1,0){\vector(1,0){3}}
\put(1,3){\vector(1,0){3}}
\put(0,2.5){\vector(0,-1){2}}
\put(5,2.5){\vector(0,-1){2}}
\end{picture}
}

\newcommand{\SquarE}[8]{
\setlength{\unitlength}{.8cm}
\begin{picture}(5,2.2)(-2.4,1.5)
\thicklines

\put(-3,3){\makebox(0,0){$#1$}}
\put(6,3){\makebox(0,0){$#2$}}
\put(-3,0){\makebox(0,0){$#3$}}
\put(6,0){\makebox(0,0){$#4$}}

\put(-3.5,1.5){\makebox(0,0)[r]{$#6$}}
\put(6.5,1.5){\makebox(0,0)[l]{$#7$}}
\put(2.5,0.5){\makebox(0,0)[c]{$#8$}}
\put(2.8,3.5){\makebox(0,0)[c]{$#5$}}

\put(1.25,0){\vector(1,0){3}}
\put(1.25,3){\vector(1,0){3}}
\put(-3,2.5){\vector(0,-1){2}}
\put(6,2.5){\vector(0,-1){2}}
\end{picture}
}

\newcommand{\SquaRE}[8]{
\setlength{\unitlength}{.8cm}
\begin{picture}(5,2.4)(-4,1.3)
\thicklines

\put(-1.5,3){\makebox(0,0){$#1$}}
\put(5,3){\makebox(0,0){$#2$}}
\put(-1.5,0){\makebox(0,0){$#3$}}
\put(5,0){\makebox(0,0){$#4$}}

\put(-2,1.5){\makebox(0,0)[r]{$#6$}}
\put(5.5,1.5){\makebox(0,0)[l]{$#7$}}
\put(2.5,0.5){\makebox(0,0)[c]{$#8$}}
\put(2.5,3.5){\makebox(0,0)[c]{$#5$}}

\put(1.25,0){\vector(1,0){3}}
\put(1.25,3){\vector(1,0){3}}
\put(-1.5,2.5){\vector(0,-1){2}}
\put(5,2.5){\vector(0,-1){2}}
\end{picture}
}

\begin{abstract}
We prove that the operad $\Big$ of natural operations on
the Hochschild cohomology has the homotopy type of the operad of
singular chains on the little disks operad. To achieve this goal, 
we introduce crossed interval groups and show 
that $\Big$ is a certain crossed interval extension of an operad $\Tam$
whose homotopy type is known. This completes the investigation of the
algebraic structure on the Hochschild cochain complex that has lasted
for several decades.
\end{abstract}


\maketitle

\begin{center}
{\it Dedicated to the memory of Jean-Louis Loday (12.1.~1946 -- 6.6.~2012)}
\end{center}

\tableofcontents

\section*{Introduction}

It is well-known that, for any `reasonable' type of algebras (where
reasonable means algebras over a quadratic Koszul operad,
see~\cite[II.3.3]{markl-shnider-stasheff:book} or the original
source~\cite{ginzburg-kapranov:DMJ94} for the terminology), 
there exists the associated
cohomology based on a `standard construction.' For example, for
associative algebras, the associated cohomology is the Hochschild
cohomology, for Lie algebras the Chevalley-Eilenberg cohomology, for
associative commutative algebras the Harrison cohomology,~\&c.

There arises a fundamental problem of describing all operations acting
on this associated cohomology, by which one usually means the
understanding of the operad of all 
natural operations
on the standard construction~\cite{markl:de}. 

In this paper we study the operad $\Big$ of all natural operations on
the Hochschild cohomology of associative algebras.  The operad $\Big$
is the totalization of a certain coloured operad $B$ with the set of
colours the natural numbers $\bbN$. We recall the intuitive definition
of $B$ from~\cite[Section~7]{markl:de} in terms of `elementary' natural
operations, and formulate also a categorical, `coordinate-free'
definition which gives a precise meaning to the naturality of
operations.  This answers a remark of Kontsevich and
Soibelman~\cite[page~26]{kontsevich-soibelman} concerning the
nonfunctoriality of the Hochschild complex on an earlier attempt by
McClure and Smith~\cite[page~2]{mcclure-smith} to define natural
operations.  We then prove, in {\bf Theorem~A}, that the two
definitions are equivalent.

In connection with the Deligne conjecture~\cite{deligne:letter},
various suboperads of $\Big$ have been studied and several results
about their homotopy types were
proved~\cite{berger-fresse,kaufmann:Top07,kontsevich-soibelman,mcclure-smith,mcclure-smith2,tamarkin-tsygan:LMP01}.\footnote{Since
this introduction was not intended as an account on the history of the
Deligne conjecture, the list is necessarily incomplete and we
apologize to everyone whose work we did not mention here.}  There is
for instance a~differential graded suboperad $\Tam$ of the operad $\Big$ whose
topological version was introduced in~\cite{tamarkin-tsygan:LMP01}.
The homotopy type of $\Tam$ was established in~\cite{BataninBerger} as
that of the operad of singular chains on the little discs operad.  The
homotopy type of a topological version of $\Tam$ was studied earlier
in~\cite{mcclure-smith2}.

The starting point of this work therefore concerned the inclusion
$\Tam \hookrightarrow \Big$. It turns out that the operad $\Big$ is,
in a certain sense, freely generated by $\Tam$. More precisely, $\Big$
is the free $\I S$-module generated by the $\I$-module $\Tam$, $\Big =
F_\S(\Tam)$, see Proposition~\ref{JarkA}.  Here $\I S$ is an analog of
the symmetric category $\Delta S$ introduced
in~\cite{fiedorowicz-loday:TAMS91} (but see
also~\cite{krasauskas:LMS87}), with the simplicial category $\Delta$
replaced by Joyal's category $\I$ of
intervals~\cite{joyal:disks}. Theorem~\ref{t2} says that the functor
$F_\S(-)$ preserves the homotopy type. Combining this feature with the
observations in the previous paragraph we obtain {\bf Theorem~B},
claiming that the operad $\Big$ has the homotopy type of the operad of
singular chains on the little disks operad $\Dis$ with the inverted
grading.  To our best knowledge, this is the first complete
description of the homotopy type of the operad $\Big$, compare also
the introduction to~\cite{markl:de}.

Besides the suboperad $\Tam$ of $\Big$ mentioned above, one has the
important McClure-Smith operad of braces $\Br$ and the non-unital
(resp.~normalized) versions $\BN$, $\TN$, $\NBr$ (resp.~ $\NormT$,
$\NormB$, $\Norm(\Br)$) of these operads.  In {\bf Theorem~C} we
relate these operads and describe their homotopy~types.

\noindent 
{\bf Conventions.\/} 
Algebraic objects in this paper live in the category of abelian
groups, i.e.~are modules over the ring $\bbZ$ of integers. Therefore,
for instance, $\otimes$ denotes the tensor product over~$\bbZ$. All
operads in this paper are {\em symmetric\/}, i.e.~with right actions
of the symmetric groups on their components. The grading (homological
on cohomological) will sometimes be indicated by $*$  in the sub- or
superscript, the simplicial and cosimplicial degrees by $\bullet$.
The paper makes ample use of the material of \cite{BataninBerger} and
\cite{bbm}.

\noindent
{\bf Acknowledgment.}  
Both authors of the paper were deeply shocked
by the tragic death of Jean-Louis Loday in June 2012. We would like to
dedicate this paper to the memory of our friend whose mathematical
talent and brilliant personality will always be remembered and
admired.  This paper would never be published without his kind
encouragement and support.  We would also like to express our thanks
to him and Zig Fiedorowicz for their readiness to share their
expertise and ideas on crossed simplicial groups with us.

Also the remarks of the referee and
editors lead to substantial improvement of the paper.

\section{Main results}
\label{Jarce_umrela_maminka.}

\def\S{\EuScript S}  \def\id{{\it id}}
\def\Ab{\EuScript Ab} \def\bbZ{{\mathbb Z}} \def\bbN{{\mathbb N}}
\def\frakM{{\mathfrak M}} \def\A{\EuScript A}
\def\hatF{{\widetilde F}} \def\sgn{{\rm sgn}}
\def\hatH{{\widetilde H}} \def\Tot{{\rm Tot\/}}
\def\hatB{{\widetilde B}}
\def\Deltaop{{\Delta^{{\it op}}}}
\def\Coch{\EuScript Coch}\def\CochChain{\EuScript {C}och\EuScript{C}hain}
\def\Ch{\EuScript Ch} \def\uCh{\underline{\EuScript Ch}}
\def\Chain{\EuScript Chain}
\def\C{\EuScript C} \def\Im{{\it Im\/}} \def\Ker{{\it Ker\/}}
\def\Set{\EuScript Set}
\def\bsigma{{\overline{\sigma}}}
\def\Rada#1#2#3{{#1_{#2},\ldots,#1_{#3}}}
\def\unit{{1 \!\! 1}} \def\whD{{\widehat D}}

In this section, which can be read independently on the rest of the
paper, we formulate our main results, Theorems~A, B and~C. Although
Theorem~B follows from Theorem~C, we decided to
state this important result separately, since Theorem~C requires more
background material and we did not want to stretch the reader's
patience more than necessary.

We open this section by defining the dg-operad $\Big = \{\Big(n)\}_{n \geq
0}$ ($\Big$ abbreviating the `big') of {\em all natural multilinear
operations\/} on the Hochschild cochain complex $\CH* = \Lin(\otexp
A*,A)$ of an associative algebra $A$ with coefficients in itself and
describe its homotopy type. In the second half we discuss various
suboperads and variants of $\Big$ including the famous operad of braces. 
We finally formulate a theorem relating these operads
and describing their homotopy types.

The dg-operad $\Big$ is the totalization of a certain $\bbN$-coloured operad
$B$ whose component $B^l_{k_1,\ldots,k_n}$ consists of natural
operations of type $(l;\Rada kln)$, $l,\Rada kln \geq 0$. 
There are several equivalent ways to say what a `natural operation'
is.  There are some `elementary' operations on the Hochschild complex,
see Definition~\ref{malo_casu}, which are meaningful for any
associative unital algebra.  One can define a coloured operad
generated by compositions of linear combinations of these operations,
see Proposition~\ref{definition-of-B}. We call this definition
{\em intuitive.\/}

Another, {\em `coordinate-free'\/} Definition~\ref{Bruslicky}, 
requires that suitably defined natural
operations must act on the endomorphisms operads of monoids $M$ in an
arbitrary additive symmetric monoidal category $V$. The Hochschild
cochain complex $\CH*$ is clearly a special case of such an
endomorphism operad with $M = A$ and $V = \Ab$, the category of
abelian groups. The naturality of operations in the second
approach has a standard categorical content and does not require
introducing any kind of `elementary' operations. By Theorem~A on
page~\pageref{sec:main-results} both definitions give isomorphic results.

We start with the intuitive definition of the operad $\B$.
Let $A$ be a unital associative algebra. 
A {\em natural operation\/} in the sense of~\cite{markl:de} is a linear
combination of compositions of the following `elementary' operations:

(a)~The insertion $\circ_i : \CH {k} \otimes \CH {l} \to \CH
{k+ l- 1}$ given, for $k,l \geq 0$ and $1 \leq i \leq k$, by the formula
\[
\circ_i(f,g)(\Rada a1{k+l-1}) :=   f(\Rada a0{i-1},
g(\Rada ai{i+l-1}),\Rada a{i+l}{k+l-1}). 
\]

(b)~Let $\mu : A \ot A \to A$ be the associative product, $\id :A \to
A$ the identity and $1 \in
A$ the unit. Then elementary operations are also
the `constants' $\mu \in \CH 2$, $\id \in \CH1$ and  $1 \in
\CH 0$.

(c)~The assignment $f \mapsto f\sigma$ permuting the inputs of a
cochain $f \in \CH k$ according to a permutation $\sigma \in S_k$ is an
elementary operation.

\begin{definition}
\label{malo_casu}
Let $\B(A)^l_{\Rada k1n}$ denote, for $l, \Rada kqn \geq 0$, the abelian
group of all natural, in the above sense, operations
\[
O : \CH{k_1} \ot \cdots \ot \CH{k_n} \to \CH l.
\]
The spaces $\B(A)^l_{\Rada k1n}$ clearly form an $\bbN$-coloured suboperad
$\B(A)$ of the endomorphism operad of the $\bbN$-coloured
collection $\{\CH n\}_{n \geq 0}$.
\end{definition}

It follows from definition that elements of $\B(A)^l_{\Rada k1n}$ can
be represented by linear combinations of $(l;\Rada k1n)$-trees in the
sense of Definition~\ref{d2b} below in which, as usual, the {\em
arity\/} of a vertex of a rooted tree is the number of its input
edges and the {\em legs\/} are the input edges of a tree,
see~\cite[II.1.5]{markl-shnider-stasheff:book} for the terminology.

\begin{definition}
\label{d2b}
Let $l,\Rada k1n$ be non-negative integers.  An $(l;\Rada k1n)$-tree
is a planar tree with legs labeled by $\rada 1l$ and three types of
vertices:
\begin{itemize}
\item[(a)] 
`white' vertices of arities $\rada {k_1}{k_n}$ labeled by 
$\rada 1n$, 
\item[(b)] 
`black' vertices of arities $\geq 2$ and
\item[(c)] 
`special' vertices of arity $0$ (no input edges).
\end{itemize}
We moreover require that there are no edges connecting two
black vertices or a black vertex with a special vertex. For $n=0$ we
allow also the exceptional trees \exeptional\ and \stub\ with no
internal vertices.  
\end{definition}

We call an internal edge whose initial vertex is special a {\em stub\/} (also
called, in~\cite{kontsevich-soibelman}, a {\em tail\/}). It follows
from definition that the terminal vertex of a stub is white; the
exceptional tree \stub\ is not a stub. An
example of an $(l;\Rada k1n)$-tree is given in Figure~\ref{fig2}.
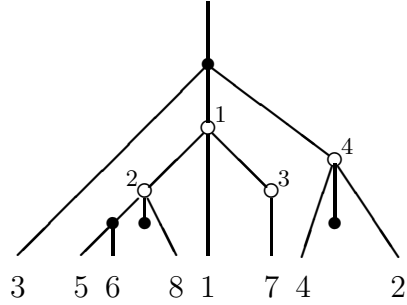
\begin{figure}
{
\unitlength=1.2pt
\thicklines
\begin{picture}(120.00,90.00)(0.00,0.00)
\put(50.00,0.00){\makebox(0.00,0.00){$8$}}
\put(80.00,0.00){\makebox(0.00,0.00){$7$}}
\put(20.00,0.00){\makebox(0.00,0.00){$5$}}
\put(30.00,0.00){\makebox(0.00,0.00){$6$}}
\put(90.00,0.00){\makebox(0.00,0.00){$4$}}
\put(0.00,0.00){\makebox(0.00,0.00){$3$}}
\put(120.00,0.00){\makebox(0.00,0.00){$2$}}
\put(60.00,0.00){\makebox(0.00,0.00){$1$}}
\put(100.00,20.00){\makebox(0.00,0.00){$\bullet$}}
\put(40.00,20.00){\makebox(0.00,0.00){$\bullet$}}
\put(60.00,70.00){\makebox(0.00,0.00){$\bullet$}}
\put(30.00,20.00){\makebox(0.00,0.00){$\bullet$}}
\put(100.00,40.00){\makebox(0.00,0.00){\large$\circ$}}
\put(102.00,42.00){\makebox(0.00,0.00)[lb]{\scriptsize $4$}}
\put(80.00,30.00){\makebox(0.00,0.00){\large$\circ$}}
\put(82.00,32.00){\makebox(0.00,0.00)[lb]{\scriptsize $3$}}
\put(40.00,30.00){\makebox(0.00,0.00){\large$\circ$}}
\put(38.00,32.00){\makebox(0.00,0.00)[rb]{\scriptsize $2$}}
\put(60.00,50.00){\makebox(0.00,0.00){\large$\circ$}}
\put(62.00,52.00){\makebox(0.00,0.00)[lb]{\scriptsize $1$}}
\put(30.00,20.00){\line(0,-1){10.00}}
\put(101.00,38.00){\line(2,-3){19.00}}
\put(100.00,39.00){\line(0,-1){19.00}}
\put(99.00,38.00){\line(-1,-3){9.50}}
\put(60.00,70.00){\line(4,-3){38.00}}
\put(80.00,28.00){\line(0,-1){18.00}}
\put(61.50,48.50){\line(1,-1){17.00}}
\put(60.00,48.00){\line(0,-1){38.00}}
\put(41.00,28.00){\line(1,-2){9.00}}
\put(40.00,28.00){\line(0,-1){8.00}}
\put(38.00,28.00){\line(-1,-1){18.00}}
\put(58.50,48.50){\line(-1,-1){17.00}}
\put(60.00,52.00){\line(0,1){18.00}}
\put(60.00,70.00){\line(-1,-1){60.00}}
\put(60.00,90.00){\line(0,-1){20.00}}
\end{picture}}
\caption{\label{fig2}
An $(8;3,3,1,3)$-tree representing an operation
in $B^8_{3,3,1,3}(A)$. It has 4 white vertices, 2 black
vertices and 2 stubs. We use the convention that directed edges
point upwards.}
\end{figure}

An $(l;\Rada k1n)$-tree $T$ determines 
the natural operation $O_T \in \B(A)^l_{\Rada k1n}$ whose action on $\CH{*}$ is
given by decorating, for each $1 \leq i \leq n$, the $i$th white
vertex by $f_i \in \CH {k_i}$, the black vertices by the iterated
multiplications, the special vertices by the unit $1$, and
performing the composition along the tree.  For instance, the tree in
Figure~\ref{fig2} represents the operation
\[
O_T(f_1,f_2,f_3,f_4)(\Rada a18) :=
a_3 f_1(f_2(a_5a_6,1,a_8),a_1,f_3(a_7))f_4(a_4,1,a_2) 
\]
where, as usual, we omit the symbol for the iteration of the 
associative multiplication $\mu$.  The exceptional $(1;)$-tree \exeptional\
represents the identity $\id \in \CH1$ and the $(0;)$-tree \stub\ the
unit $1 \in \CH0$.

\begin{notation}
For each $l,\Rada k1n \geq 0$ denote by $B^l_{\Rada k1n}$ the free
abelian group spanned by all $(l;\Rada k1n)$-trees. The
correspondence $T \mapsto O_T$ defines,  for each associative algebra 
$A$, a linear epimorphism $\omega_A : B^l_{\Rada k1n} \epi B(A)^l_{\Rada k1n}$.
\end{notation}

Let $T'$ be an $(l';\Rada {k'}1n)$-tree, $T''$ an $(l'';\Rada
{k''}1m)$-tree and assume that $l''= k'_i$ for some $1 \leq i \leq n$.
The {\em $i$th vertex insertion\/} assigns to
$T'$ and $T''$ the tree $T' \circ_i T''$ obtained by replacing the
white vertex of $T'$ labelled $i$ by $T''$. It may happen that
this replacement creates edges connecting black vertices. In that
case it is followed by collapsing these edges. The above construction
extends into a linear operation
\begin{equation}
\label{eq:1}
\circ_i : B^{l'}_{\Rada {k'}1n} \otimes B^{l''}_{\Rada {k''}1m} 
\to B^{l'}_{\Rada {k'}1{i-1},\Rada {k''}1m,\Rada {k'}{i+1}n},\
1 \leq i \leq n,\ l''= k'_i.
\end{equation}

\begin{proposition}
\label{ani_mi_neodpovedela}\label{definition-of-B}
The spaces $B^l_{\Rada k1n}$ assemble into an 
$\bbN$-coloured operad $B$ with the operadic composition given by the vertex
insertion and the symmetric group relabeling the white vertices. 
With this operad structure, the maps $\omega_A : B^l_{\Rada k1n} \epi
B(A)^l_{\Rada k1n}$  form an epimorphism $\omega_A : B \epi
B(A)$ of $\bbN$-coloured operads.
\end{proposition}

\begin{proof}
A direct verification.
\end{proof}

\begin{definition}
A unital associative algebra $A$ is {\em generic\/} if the
map $\omega_A : B \epi B(A)$ is an isomorphism.
\end{definition}

Theorem~\ref{snad_mne_neopusti} proved in 
Appendix~\ref{Jarunka_jeste_spinka} states that free
associative algebras generated by countably many generators are generic,
so generic algebras exist.  The operad $B$ can thus be
equivalently defined as the operad of elementary operations acting on
the Hochschild complex of a~generic algebra.  This point of view will
be used in our description of the isomorphism between the underlying
category of the coloured operad $B$ and the linearization of the
category $(\I S)^{\op}$ given in Section~\ref{JARka}.

Let us proceed to the `coordinate-free' definition of the `big'
operad.  Denote by $\SYMCAT$ the category whose objects are pairs
$(V,M)$ of an additive strict symmetric monoidal category $V$ and a
unital monoid $M$ in $V$. Morphisms are additive strict symmetric
monoidal functors preserving distinguished monoids.\footnote{Using
other types of symmetric monoidal categories and symmetric monoidal
functors leads, however, to the same result.}  It is interesting to
note that the initial object of $\SYMCAT$ is the category $(\Delta
S)_{alg}$ recalled on page~\pageref{moz}, with $1$ as the
distinguished monoid.

\begin{definition} 
\label{Bruslicky}
Let $\NAT^l_{k_1,\ldots,k_n}$ be, for $l,\Rada k1n \geq 0$, the abelian group
of families of linear maps
\begin{equation}
\label{Kopou_v_baraku.} 
\alpha_{(V,M)} : \underline{V}(M^{\ot k_1},M)\otimes \cdots \otimes
\underline{V}(M^{\ot k_n},M)\rightarrow \underline{V}(M^{\ot l},M)
\end{equation} 
indexed by objects $(V,M)\in \SYMCAT$ which are natural in the sense
that, for any morphism $F:(V,M)\rightarrow (W,N)$, the diagram
\begin{equation}
\label{Za_chvili_na_bruslicky.}
\SquarE
{\underline{V}(M^{\ot k_1},M)\otimes \cdots \otimes
  \underline{V}(M^{\ot k_n},M)}{\underline{V}(M^{\ot l},M)}
{\underline{W}(N^{\ot k_1},N)\otimes \cdots \otimes
  \underline{W}(N^{\ot k_n},N)}{\underline{W}(N^{\ot l},N)}
{\alpha_{(V,M)}}{F\ot \cdots \ot F}{F}{\alpha_{(W,N)}}
\raisebox{-4em}{\rule{0pt}{0pt}}
\end{equation}
commutes. The spaces $\NAT^l_{k_1,\ldots,k_n}$ clearly assemble into a
coloured operad $\NAT$ in the category $\Ab$ of abelian groups, with
natural numbers $\bbN = \{0,1,2,\ldots\}$ as its set of colors.
\end{definition}  

The following first main result of our paper is proved in
Section~\ref{Zavalen_krabicemi}.

\begin{theoremA}
\label{sec:main-results}
The $\bbN$-coloured operads $\NAT$ and $\B$ are canonically isomorphic.
\end{theoremA}

\begin{remark}
The isomorphism $\NAT \cong \B$ 
is a particular feature of the (multi)linearity implied 
by the $\Ab$-enrichment used in this paper.
The operad $\NAT$ is, for general enrichments, 
bigger than $\B$. For instance,
in the Cartesian situation it contains also the projections
\[
\pi^i_{(V,M)} : \underline{V}(M^{\ot k_1},M)\times \cdots \times
\underline{V}(M^{\ot k_n},M)\rightarrow \underline{V}(M^{\ot k_i},M),
\ 1 \leq i \leq n.
\]
\end{remark}

The operads $\B$ and $\NAT$ have obvious nonsymmetric analogs. More
precisely, let $\check{T}$ be the $\bbN$-coloured operad obtained by
taking in Definition~\ref{Bruslicky} instead of $\Symcat$ the category
of pairs $(V,M)$ consisting of a {\em nonsymmetric\/} monoidal
category $V$ and a unital monoid $M$ in~$V$. Likewise, let ${T}$ be
the result of removing (c) from the list of elementary operations used
in Definition~\ref{malo_casu}.\footnote{This operad was first
described in~\cite[Section~3]{tamarkin-tsygan:LMP01}. Still another
presentation of $T$ can be found in~\cite[Proposition~4.10]{bbm} or in
the proof of \cite[Proposition~2.14]{BataninBerger}.}  The following analog
of Theorem~A holds.

\begin{theorem}
\label{Za_ctyri_dny_uvidim_Jarku.}
The $\bbN$-coloured operads $T$ and $\check{T}$ are canonically isomorphic.
\end{theorem}

The operad $B$ has, as each $\bbN$-colored operad in $\Ab$, its {\em
underlying category\/} $\calU(B)$ enriched in~$\Ab$, whose objects are
natural numbers and the enriched Hom-sets are $\calU(B)(k,l) :=
B^l_k$.  The $\circ_i$-operation of~(\ref{eq:1}) taken with $i=
n=m=1$, $l' = l$, $l'' = k'_1 = k$ and $k''_1= m$, i.e.~a map $B^l_k
\ot B^k_m \to B^l_m$, gives the categorial composition.

Recall that the simplicial category $\Delta$ 
has objects the finite ordinals $[n] = \{\rada
0n\}$, $n \geq 0$.  Morphism of $\Delta$ are generated by the face
maps $\delta_i : [n-1] \to [n]$ (misses $i$) and the degeneracy maps
$\sigma_i : [n+1] \to [n]$ (hits $i$ twice), $i = \rada0n$.
There is a natural functor $u: \Delta \to  \calU(B)$ such that $u([n])
:= n$, $u(\delta_i) \in B^n_{n-1}$
is defined as 
\[
\raisebox{-1.5em}{}
u(\delta_0) := \lwhite, \hskip 1em  u(\delta_i) := \cwhite, \
1 \leq i \leq n-1,\ \mbox { and }\ 
u(\delta_{n}) := \hskip .9em \rwhite,
\]
and $u(\sigma_i) \in B^n_{n+1}$ as
\[
\raisebox{-1.5em}{}
u(\sigma_i) := \cwhitestub, \ 0 \leq i \leq n.
\]
In the displays, $n-1$, resp.~$n+1$, denotes the arity of the white
vertex and $i$ indicates its $i$th input. The above trees obviously act on
the Hochschild complex of an associative algebra via the standard
cosimplicial structure.

It follows from general properties of coloured operads that $B^l_k$
acts on $B^l_{\Rada k1n}$ covariantly on the upper index and
contravariantly on each lower index.  Therefore, the spaces
$B^l_{\Rada k1n}$ assemble into a functor
\begin{equation}
\label{eq:2}
B^{\bullet}_{\bullet_1,\ldots,\bullet_n} :\Delta\times
(\Delta^{\op})^n\to \Ab.
\end{equation}

We can totalize the functor $B^{\bullet}_{\bullet_1,\ldots,\bullet_n}$
(i.e.~we apply the cosimplicial $\uTot$ with respect to the upper index and
multisimplicial $\lTot$ with respect to the lower
indices).\footnote{The cosimplicial totalization is recalled in
  Section~\ref{jarKA}, the (multi)simplicial totalization is the standard one.}
That is, for  any $n \geq 0$ we  put\label{zacatek}
\[
\Big^*(n) := \uTot(\lTot B^{\bullet}_{\bullet_1,\ldots,\bullet_n})
= \prod_{l - (k_1 + \cdots + k_n) = *} {B^l_{\Rada k1n}}. 
\]
The degree $+1$ differential $d : \Big^*(n) \to \Big^{*+1}(n)$ is thus given
by $d : = \delta - (\pa_1 + \cdots + \pa_n)$, where $\delta$
(resp.~each $\pa_i$, $1 \leq i \leq n$) is induced from the
corresponding cosimplicial (resp.~simplicial) structure. 

It is not hard to observe (see the appendix of \cite{bbm} for a formal
categorical proof) that the collection $\Big = \{\Big ^*(n)\}_{n \geq
0}$ is a dg operad which acts on the cosimplicial totalization $\CH *
= \uTot(\CH{\bullet})$ of the Hochschild cochains with the standard
structure of the cosimplicial abelian group. We remind the reader
that, according to our conventions, $\CH *$ denotes the Hochschild
cochain complex while $\CH{\bullet}$  the
cosimplicial abelian group of the Hochschild cochains.

The structure of the operad $\Big$ is visualized in
Figure~\ref{jarka1}. The degree $m$-piece of $\Big(n)$ is the
direct {\em product\/}, not the direct sum, of elements on the
diagonal $p+q=m$. Therefore the usual spectral sequence arguments do
not apply. For instance, it can be shown that all rows in
Figure~\ref{jarka1} are acyclic, but $\Big(n)$ is not acyclic! We have

\begin{figure}
{
\unitlength=1.2pt
\begin{picture}(200.00,173.00)(0.00,40.00)
\put(150.00,52){\makebox(0.00,0.00){$\vdots$}}
\put(100.00,52){\makebox(0.00,0.00){$\vdots$}}
\put(50.00,52){\makebox(0.00,0.00){$\vdots$}}
\put(200.00,80.00){\makebox(0.00,0.00){$\cdots$}}
\put(200.00,110.00){\makebox(0.00,0.00){$\cdots$}}
\put(200.00,140.00){\makebox(0.00,0.00){$\cdots$}}
\put(200.00,170.00){\makebox(0.00,0.00){$\cdots$}}
\put(200.00,200.00){\makebox(0.00,0.00){$\cdots$}}
\put(150.00,120.00){\vector(0,1){10.00}} 
\put(100.00,90.00){\vector(0,1){10.00}}
\put(50.00,90.00){\vector(0,1){10.00}}
\put(150.00,90.00){\vector(0,1){10.00}}
\put(100.00,60.00){\vector(0,1){10.00}}
\put(50.00,60.00){\vector(0,1){10.00}}
\put(150.00,60.00){\vector(0,1){10.00}}
\put(100.00,120.00){\vector(0,1){10.00}}
\put(50.00,120.00){\vector(0,1){10.00}}
\put(150.00,150.00){\vector(0,1){10.00}}
\put(100.00,150.00){\vector(0,1){10.00}}
\put(50.00,150.00){\vector(0,1){10.00}}
\put(150.00,180.00){\vector(0,1){10.00}}
\put(100.00,180.00){\vector(0,1){10.00}}
\put(50.00,180.00){\vector(0,1){10.00}}
\put(165.00,140.00){\vector(1,0){25.00}}
\put(115.00,140.00){\vector(1,0){20.00}}
\put(65.00,140.00){\vector(1,0){20.00}}
\put(10.00,140.00){\vector(1,0){25.00}}
\put(165.00,110.00){\vector(1,0){25.00}}
\put(115.00,110.00){\vector(1,0){20.00}}
\put(65.00,110.00){\vector(1,0){20.00}}
\put(10.00,110.00){\vector(1,0){25.00}}
\put(165.00,80.00){\vector(1,0){25.00}}
\put(115.00,80.00){\vector(1,0){20.00}}
\put(65.00,80.00){\vector(1,0){20.00}}
\put(10.00,80.00){\vector(1,0){25.00}}
\put(10.00,170.00){\vector(1,0){25.00}}
\put(165.00,170.00){\vector(1,0){25.00}}
\put(115.00,170.00){\vector(1,0){20.00}}
\put(65.00,170.00){\vector(1,0){20.00}}
\put(160.00,200.00){\vector(1,0){30.00}}
\put(110.00,200.00){\vector(1,0){30.00}}
\put(60.00,200.00){\vector(1,0){30.00}}
\put(150.00,80.00){\makebox(0.00,0.00){$B^2_3(n)$}}
\put(150.00,110.00){\makebox(0.00,0.00){$B^2_2(n)$}}
\put(150.00,140.00){\makebox(0.00,0.00){$B^2_1(n)$}}
\put(150.00,170.00){\makebox(0.00,0.00){$B^2_0(n)$}}
\put(100.00,80.00){\makebox(0.00,0.00){$B^1_3(n)$}}
\put(100.00,110.00){\makebox(0.00,0.00){$B^1_2(n)$}}
\put(100.00,140.00){\makebox(0.00,0.00){$B^1_1(n)$}}
\put(100.00,170.00){\makebox(0.00,0.00){$B^1_0(n)$}}
\put(50.00,80.00){\makebox(0.00,0.00){$B^0_3(n)$}}
\put(50.00,110.00){\makebox(0.00,0.00){$B^0_2(n)$}}
\put(50.00,140.00){\makebox(0.00,0.00){$B^0_1(n)$}}
\put(50.00,170.00){\makebox(0.00,0.00){$B^0_0(n)$}}
\put(0.00,80.00){\makebox(0.00,0.00){$0$}}
\put(0.00,110.00){\makebox(0.00,0.00){$0$}}
\put(0.00,140.00){\makebox(0.00,0.00){$0$}}
\put(0.00,170.00){\makebox(0.00,0.00){$0$}}
\put(150.00,200.00){\makebox(0.00,0.00){$0$}}
\put(100.00,200.00){\makebox(0.00,0.00){$0$}}
\put(50.00,200.00){\makebox(0.00,0.00){$0$}}
\end{picture}}
\caption{%
The structure of the big operad $\Big$. In the above
diagram, $B^m_k(n) := \prod_{k_1+\cdots +k_n =k}
B^m_{k_1,\ldots,k_n}$. The vertical arrows are the simplicial
differentials $\pa$ and the horizontal arrows are the cosimplicial
differentials~$\delta$.\label{jarka1}}
\end{figure}

\begin{theoremB}
The big dg-operad $\Big$ of all natural operations on the Hochschild
complex has the homotopy type of the dg-operad $\SC_{-*}(\Dis)$ of
singular chains on the little disks operad $\Dis$ with the inverted
grading.
\end{theoremB}

Theorem~B is a consequence of Theorem~C formulated at the end of this section. 

\begin{variants}
\label{Jarunka}
An important suboperad of $\Big$ is the suboperad $\BN$ spanned
by trees {\em without\/} stubs and without the exceptional tree
\stub. The operad $\BN$ is the operad of
all natural multilinear operations on the Hochschild complex of a~{\em
non-unital\/} associative algebra. It is generated by natural
operations (a)--(c) above but without the unit $1 \in \CH 0$ in (b).

As we already know, the arity $n$-component $\Big(n)$ of the operad
$\Big$ is, for each $n \geq 0$, an $n$-times simplicial abelian
group. One can therefore consider its normalization 
$\Norm(\Big)(n)$, i.e.~the quotient of $\Big(n)$ by images of the
simplicial degeneracies. Since the simplicial degeneracies of
the $i$th simplicial structure, $1 \leq i \leq n$, 
act by growing stubs from the white vertex labelled $i$, we see that,
as dg-abelian groups,
\[
\Norm(\Big)(n) \cong \cases{\BN(n)}{for $n \geq 1$, and}{\Big(0)}{for $n=0$.}
\]
The space $\BN(0)$ is a proper subset of
$\NormB(0)$ as the latter contains also
the span of the `exceptional tree' \stub.
It is not difficult to see that the componentwise normalizations
$\{\NormB(n)\}_{n \geq 0}$ assemble into an operad $\NormB$. One may
also identify $\NormB$ with the subcollection of $\Big$ spanned
by trees without stubs, but keeping in mind that
$\NormB$ is not a suboperad of $\Big$.  

The operad $\BN$ is the same as the operad introduced, under the name
$\Big_{\sAss}$, in~\cite{markl:de}. In the notation of that paper,
$\Big$ would be denoted by $\Big_{\sUAss}$. It was shown
by direct calculations, in~\cite[Example~12]{markl:de}, that
$\Big_{\sAss}(0) \cong \BN(0)$ are acyclic dg-abelian groups.

While $\BN$ acts on the Hochschild complex of a non-unital
algebra, the operad $\NormB$ acts on the {\em normalized\/} Hochschild
complex of a {\em unital\/} algebra. One has the chain of
operad maps
\begin{equation}
\label{koleno_porad_boli}
\BN \stackrel\iota\hookrightarrow \Big \stackrel\pi\twoheadrightarrow \NormB,
\end{equation}
in which the projection $\pi$ is a weak equivalence and the components
$\pi\iota(n)$ of the composition $\pi\iota$ are isomorphisms for each
$n \geq 1$. If $\mathfrak{U}$ denotes the functor that replaces the arity
zero component of a dg-operad by the trivial abelian group, then
$\mathfrak{U}(\pi\iota)$ is clearly a dg-operad isomorphism $\mathfrak{U}(\BN) \cong
\mathfrak{U}(\NormB)$.
\end{variants}

\noindent 
{\bf The operad $\Tam .$}
Totalizing the operad $T$ of Theorem~\ref{Za_ctyri_dny_uvidim_Jarku.} 
in the same way as $B$ we produce a suboperad $\Tam$ of $\Big$ whose
 arity-$n$ piece equals
\[
\Tam^*(n) :=\uTot(\lTot T^{\bullet}_{\bullet_1,\ldots,\bullet_n})^*
= \prod_{l - (k_1 + \cdots + k_n) = *} {T^l_{\Rada k1n}}, 
\]
where operations in $T^l_{\Rada k1n}$ are represented by linear
combinations of {\em unlabeled\/} $(l;\Rada k1n)$-trees, 
that is, planar trees as
in Definition~\ref{d2b} but without the labels of the legs. The
inclusion $T^l_{\Rada k1n} \hookrightarrow B^l_{\Rada k1n}$ is realized by 
labeling the legs of an unlabeled tree counterclockwise in the
orientation given by the planar embedding. The operad
$\Tam$ is a chain version of the operad considered 
in~\cite[Section~3]{tamarkin-tsygan:LMP01}.\label{konec} 

There are finally the operad $\TN :=
\BN \cap \Tam$ generated by unlabeled trees without stubs and without
\stub, and the normalized quotient $\NormT$ of $\Tam$. As before, one
has the diagram 
\[
TN \stackrel\iota\hookrightarrow \Tam
\stackrel\pi\twoheadrightarrow \NormT
\] 
with the properties analogous to that of~(\ref{koleno_porad_boli}).

\noindent 
{\bf Operads of braces.}
\label{Kveta_prespala_v_Zitne.}
There is another very important suboperad $\Br$ of $\Big$ generated by braces,
cup-products and the unit.
More precisely, $\Br$ is the suboperad of the big operad 
$\Big$ generated by the following operations.

(a) The {\em cup product\/} $\CH * \ot \CH * \to \CH *$ given, 
for $f \in \CH k$ and $g \in \CH l$, by
\[
(f \cup g)(\Rada a1{k+l}):= f(\Rada a1k)g(\Rada a{k+1}{k+l}).
\]

(b) The constant $1 \in \CH 0$.

(c) The {\em braces\/} $-\{\rada --\} : 
\otexp {\CH *}n \to \CH*$, $n \geq 2$, given by
\[
f\{\Rada g2n\} := \sum (-1)^\epsilon f(\rada {\id}{\id},g_2,\rada
{\id}{\id},g_n,\rada {\id}{\id}),
\]
where $\id$ is the identity map of $A$, the summation runs over all
possible substitutions of $\Rada g2n$ (in that order) into $f$ and
\[
\epsilon := \sum_{j=2}^n (\deg(g_i)-1)t_j,
\]
in which $t_j$ is the total number of inputs in front of $g_j$.

The brace operad has also its non-unital version $\NBr := \BN \cap
\Br$ generated by elementary operations (a) and (c). One can verify
that both $\Br$ and $\NBr$ are indeed dg-suboperads of $\Big$,
see~\cite{bbm,mcclure-smith}. 

To complete the picture, we denote by
$\Norm(\Br) \subset \Norm(\Big)$ the image of $\Br$ under the projection
$\Big \epi \NormB$. One has again an analog 
$\NBr \stackrel\iota\hookrightarrow \Br \stackrel\pi\twoheadrightarrow
\Norm(\Br)$ of~(\ref{koleno_porad_boli}).

It turns out that the piece $\Br(n)$ of the operad $\Br$ is, for each
$n\geq 0$, the unnormalized chain complex associated to the $n$-times
simplicial abelian group \hbox{$T(\Rada \bullet 1n;0): (\Delta^{\rm
op})^{\times n} \to \Ab$}, see~\cite[Section~3.1]{bbm} for details. 
The chain complex $\Norm(\Br)(n)$ is then the
simplicial normalization of $T(\Rada \bullet 1n;0)$, 
so the map $\Br \stackrel\pi\twoheadrightarrow \Norm(\Br)$ is
a weak equivalence as in the previous cases.

\vskip .3em
\noindent 
{\bf Remark.} 
The operads $\Norm(\Br)$ and $\NBr$ were first introduced by McClure and Smith~\cite{mcclure-smith} under the names
 $\mathcal H$ and $\mathcal H'$, respectively. 
 
\begin{figure}
\[
{
\thicklines
\unitlength=.7pt
\begin{picture}(360.00,263.00)(-70.00,0.00)
\put(200.00,100.00){\makebox(0.00,0.00){$\BN$}}
\put(160.00,0.00){\makebox(0.00,0.00){$\TN$}}
\put(0.00,50.00){\makebox(0.00,0.00){$\NBr$}}
\put(160.00,80.00){\makebox(0.00,0.00){$\Tam$}}
\put(200.00,180.00){\makebox(0.00,0.00){$\Big$}}
\put(0.00,130.00){\makebox(0.00,0.00){$\Br$}}
\put(0,83){
\put(160.00,80.00){\makebox(0.00,0.00){$\NormT$}}
\put(200.00,177.00){\makebox(0.00,0.00){$\NormB$}}
\put(0.00,130.00){\makebox(0.00,0.00){$\Norm(\Br)$}}
}
\put(170.00,92.00){\vector(4,1){20.00}}
\put(24.00,55.00){\line(4,1){100.00}}
\put(0,83){
\put(170.00,92.00){\vector(4,1){20.00}}
\put(24.00,55.00){\line(4,1){100.00}}
}
\put(24.00,124.00){\vector(3,-1){120.00}}
\put(0,83){
\put(30.00,122.00){\vector(3,-1){97.00}}
}
\put(24.00,44.00){\vector(3,-1){120.00}}
\put(32.00,223.00){\vector(4,1){135.00}}
\put(0.00,68.00){\vector(0,1){49.00}}
\put(-5.00,63.00){
\qbezier(0,10)(0,0)(2.5,0)\qbezier(2.5,0)(5,0)(5,10)
}
\put(160.00,18.00){\vector(0,1){49.00}}
\put(155,13.00){
\qbezier(0,10)(0,0)(2.5,0)\qbezier(2.5,0)(5,0)(5,10)
}
\put(200.00,122.00){\vector(0,1){45.00}}
\put(195,118.00){
\qbezier(0,10)(0,0)(2.5,0)\qbezier(2.5,0)(5,0)(5,10)
}
\put(0,83){
\put(0.00,63.00){\vector(0,1){54.00}}
\put(0.00,63.00){\vector(0,1){44.00}}
\put(160.00,13.00){\vector(0,1){54.00}}
\put(160.00,13.00){\vector(0,1){44.00}}
\put(200.00,113.00){\vector(0,1){54.00}}
\put(200.00,113.00){\vector(0,1){44.00}}
}
\put(165.00,103.00){\vector(1,2){33.00}}
\put(166.00,15.00){\vector(1,2){33.00}}
\put(0,83){
\put(161.00,93.00){\vector(1,2){35.00}}
}
\put(-5,170){\makebox(0.00,0.00)[r]{$\pi$}}
\put(206,220){\makebox(0.00,0.00)[l]{$\pi$}}
\put(155,120){\makebox(0.00,0.00)[r]{$\pi$}}

\put(0,-80)
{
\put(-5,170){\makebox(0.00,0.00)[r]{$\iota$}}
\put(206,220){\makebox(0.00,0.00)[l]{$\iota$}}
\put(155,120){\makebox(0.00,0.00)[r]{$\iota$}}
}

\end{picture}}
\]
\caption{Operads of natural operations and their maps. The horizontal
maps are inclusions.\label{fig3}}
\end{figure}
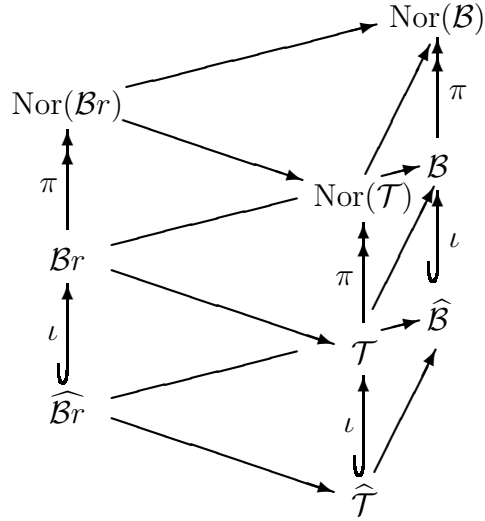

Recall that $\mathfrak{U}$ is the functor that replaces the arity
zero component of a dg-operad by the trivial abelian group.
The last main result of this paper is:

\begin{theoremC} 
The operads introduced above can be organized into the diagram
in Figure~\ref{fig3} in which:
\begin{enumerate}
\item  
Operads in the two upper triangles have the chain homotopy type of the operad
$\SC_{-*}(\Dis)$ of singular chains on the little disks operad $\Dis$
with the inverted grading.

\item  
all  morphisms  between vertices of the two upper triangles 
are weak equivalences,
  
\item operads in the bottom triangle of Figure~\ref{fig3} have the chain
homotopy type of the operad $\SC_{-*}(\Dis)$ with the component of
arity $0$ replaced by the trivial abelian group, and 

\item  all 
morphisms in Figure~\ref{fig3} become weak equivalences after the 
application of the functor~$\mathfrak{U}$. 
\end{enumerate}
\end{theoremC}

\noindent 
Theorem~C is proved in Section~\ref{JARka}.

\begin{center}
{\bf Plan for the proofs.} 
\end{center}

\noindent 
{\bf Section~\ref{Zavalen_krabicemi}.\/} To prove Theorem~A that states
that both the intuitive and categorical definitions of the ${\mathbb
N}$-coloured operad~$B$ lead to isomorphic results, we introduce, in
Definition~\ref{hh}, an auxiliary ${\mathbb N}$-coloured
operad $\Nat$ of natural operations acting
on multiplicative operads with unit. Theorem~A will then follow from
Propositions~\ref{Pisu_v_Koline} and~\ref{Jarka_si_zapomela_ruzicku.}
claiming that both definitions of $B$ give operads isomorphic to
$\Nat$.

\noindent 
{\bf Section~\ref{2_den}.\/} We introduce crossed interval groups and
describe our main example, the crossed interval group $\S$ with the
underlying category $\I S$.  For a crossed
interval group $\H$, we define $\H$-objects in a category $\C$ and
describe, in Lemma~\ref{freeHobject}, the left adjoint $F_{\H}$  to the
restriction functor from the category of $\H$-objects in $\C$ to the
category of cosimplicial objects in $\C$.

\noindent 
{\bf Section~\ref{jarKA}.\/} This section is devoted to the particular
case of $\S$-objects in the category $\Chain$ of chain complexes.  We
use an acyclic-models-type argument developed in
Appendix~\ref{sec:acycl-models-theor} to prove, in Theorem~\ref{t2},
that the functor $F_\S$ preserves the homotopy type.

\noindent 
{\bf Section~\ref{JARka}.\/} We show that the pieces of the coloured
operad $B$ assemble into a functor from $(\I S)^{\op} \times (\I
S)^{\times n}$ to the category of abelian groups.  By
Proposition~\ref{JarkA}, this functor is the result of the application
of the functor $F_{\S}$ to the
pieces of a coloured operad whose homotopy type is known. This, along
with Theorem~\ref{t2} and a couple of standard arguments, implies
Theorem~C.

\noindent
{\bf Appendix~\ref{sec:acycl-models-theor}.\/} We introduce a theory
of acyclic models for functors with values in cochain complexes. The
main result, Theorem~\ref{p2}, is a dual version of the classical
(chain) Acyclic Models Theorem.

\noindent 
{\bf Appendix~\ref{Jarunka_jeste_spinka}.\/} The concept of generic
algebras introduced in Section~\ref{Jarce_umrela_maminka.} allows
to view elements of the operad $B$ as multilinear operations on the Hochschild
complex of a concrete algebra. In this part of the appendix we show
that free associative unital algebras generated by countably many
generators are~generic.

\section{Multiplicative operads and proof of Theorem~A}
\label{Zavalen_krabicemi}

The aim of this section is to prove Theorem~A and necessary auxiliary
results. We start by analyzing the structure of the
operad $\B$, using mostly the material taken from~\cite{bbm}
and~\cite{markl:de}.  We then prove, in
Proposition~\ref{Pisu_v_Koline}, that the operad $\NAT$ is isomorphic
to an `auxiliary' operad $\Nat$ of natural operations acting on
multiplicative operads with unit. Finally, we show in
Proposition~\ref{Jarka_si_zapomela_ruzicku.} that $\B \cong
\Nat$. This will imply Theorem~A.

The identification of $\B^l_{\Rada k1n}$ with the span of the set of 
$(l;\Rada k1n)$-trees offers another description of
the pieces of the operad $\B$. Denote by $\Fr$ the quotient
\begin{equation}
\label{Zitra_je_pohreb_Jerusciny_maminky.}
\Fr := \frac{\Gamma(\Rada o1n,\mu,e)}{(\mu \circ_1 \mu = \mu \circ_2
  \mu,\ \mu \circ_1 e = \mu \circ_2 e = \id)},
\end{equation}
of the free operad $\Gamma(\Rada o1n,\mu,e)$ 
generated by $n$ operations $o_i \in \Fr(k_i)$ of arities $k_i$,
$1 \leq i \leq n$, one operation $\mu$ of arity $2$ and one operation
$e$ of arity $0$, modulo the ideal generated by the
associativity $\mu \circ_1 \mu = \mu \circ_2 \mu$ and unitality $\mu
\circ_1 e = \mu \circ_2 e = \id$ of
$\mu$.

Let $\Fr_{\rada 11}(l)$ be the subspace of
elements of $\Fr(l)$ containing each generator $\Rada o1n$  precisely once.
The structure theorem 
for free operads~\cite[Section~4]{markl:handbook} immediately gives:

\begin{proposition}
\label{psano_na_letisti}
For each $l, \Rada k1n \geq 0$ and $\Fr_{\rada 11}(l)$ as above, one has the
canonical isomorphism of abelian groups
\[
\B^l_{\Rada k1n} \cong \Fr_{\rada 11}(l).
\]
\end{proposition}

Let  $\UAss$ be the operad for unital associative algebras.

\begin{definition}
\label{Jarunka_pusa.}
\hskip -.4em\footnote{Multiplicative {\it nonsymmetric} operads were defined
in \cite{GV}. } 
A {\em multiplicative operad with unit\/} is an operad $\O$ together with a
morphism $p: \UAss \to \O$, i.e.~an operad under
$\UAss$. Equivalently, $\O$ has a distinguished `multiplication' $\mu \in
\O(2)$ and a `unit' $e \in \O(0)$ 
such that $\mu \circ_1 \mu = \mu \circ_2 \mu \in \O(3)$ and $\mu
\circ_1 e = \mu \circ_2 e = \id \in \O(1)$.
\end{definition}

Let $\MultSO$ be the category of multiplicative operads in $\Ab$ with
unit, i.e.~the comma-category $\UAss/\SO$, where $\SO$ is the category
of (symmetric) operads.  Objects of $\MultSO$ are therefore couples
$(\O,p)$ of an (ordinary) operad and the structure map. We will,
however, drop the structure map from the notation since it will always
be clear from the context whether we mean a multiplicative operad with
unit or an ordinary operad.  Proposition~4.11 of~\cite{bbm} gives an
alternative characterization of the category $\MultSO$:

\begin{proposition}[\cite{bbm}]
\label{na_letisti}
The category of algebras over the
$\bbN$-coloured operad $\B$ is isomorphic to the category of
multiplicative operads with unit.
\end{proposition}

\begin{remark}The operad $\Fr = (\Fr,\mu,{\bf 1})$ is clearly the
free multiplicative operad with unit generated by $o_i \in \Fr(k_i)$, $1
\leq i \leq n$.
Proposition~\ref{psano_na_letisti} can also be deduced from
Proposition~\ref{na_letisti} and the standard relation between free algebras
of a certain type and the operad governing this type of algebras.
\end{remark}

We will show that $\NAT$ is isomorphic to a certain operad of natural
operations acting on objects of $\MultSO$.

\begin{definition} 
\label{hh}
Let $\Nat^l_{k_1,\ldots,k_n}$ be, for $l,\Rada k1n \in \bbN$, the abelian group
of linear maps
\[ 
\beta_{\omu} : 
\O(k_1) \ot \cdots \ot \O(k_n) \to \O(l)
\] 
indexed by multiplicative operads with unit $\omu \in \MultSO$ such
that, for any morphism $\varphi:\omu \rightarrow \pnu$ of
multiplicative operads with unit, the diagram
\begin{equation}
\label{Kolecko}
\hskip -4cm
\SquaRE
{\O(k_1) \ot \cdots \ot \O(k_n)}{\O(l)}
{\P(k_1) \ot \cdots \ot \P(k_n)}{\P(l)}
{\beta_{\omu}}{\varphi(k_1)\ot \cdots \ot\varphi(k_n)}
{\varphi(l)}{\beta_{\pnu}}
\raisebox{-4em}{\rule{0pt}{0pt}}
\end{equation}   
commutes. The abelian groups $\Nat^l_{k_1,\ldots,k_n}$ form
an $\bbN$-coloured operad $\Nat$.
\end{definition}  

\begin{proposition}
\label{Pisu_v_Koline}
The $\bbN$-coloured operads $\Nat$ and $\NAT$ are naturally isomorphic.
\end{proposition}

The proof makes use of a relation between the categories $\SYMCAT$ and $\MultSO$.
Each object $(V,M) \in \SymCat$ has its {\em endomorphism
operad\/} 
\[
\End_{(V,M)} := \{\End_{(V,M)}(n)\}_{n \geq 0} \ \ \mbox{where} \ \ 
\End_{(V,M)}(n) := \underline V (M^{\ot n},M),
\]
with the convention that $M^{\ot 0} := {\bf 1}$, the unit object of the
monoidal category $V$. 
The unital monoid structure on $M$ clearly determines an operad map 
$p:\UAss \to \End_{(V,M)}$ so that  $\End_{(V,M)} \in \MultSO$.
It is not hard to check that this construction extends to a functor
\[
\End:\SymCat \to \MultSO.
\]
Observe that $\End((\Delta S)_{alg}) \cong \UAss$.
With this notation, the operation $\alpha_{(V,M)}$ in~(\ref{Kopou_v_baraku.})
is a map
\[
\alpha_{(V,M)} : \End_{(V,M)}(k_1) \ot \cdots \ot \End_{(V,M)}(k_n)
\to \End_{(V,M)}(l).
\]

\begin{lemma}
\label{Smutek} 
The functor $\End$ has a left adjoint 
\[
P: \MultSO \rightarrow \SymCat.
\]
The unit of this adjunction is the identity and the counit
\[
U : 
(P(\End_{(V,M)}),1) \to (V,M)
\]
is the strict monoidal functor whose effect on the Hom-sets is the identity. 
Moreover, the functor $\End$ is surjective on morphisms (but not full). 
\end{lemma}

\begin{remark}
A functor $F : \C' \to \C''$ is surjective on morphism if for any
objects $a'',b'' \in \C''$ and a morphism $f'' \in \C''(a'',b'')$ there exist
objects $a',b' \in \C'$  and a morphism $f' \in \C'(a',b')$ such that $F(a') =
a''$, $F(b') = b''$ and $f'' = F(f')$. On the other hand, being full
means that, for any objects  $a',b' \in \C'$, the map $F :
\C'(a',b') \to \C'(F(a'),(b'))$ is an epimorphism. These two
conditions are obviously different.

\end{remark}

\begin{proof}[Proof of Lemma~\ref{Smutek}]  
For an operad $\O$, let $P(\O)$ be the PROP freely generated by $\O$,
i.e.~the additive symmetric monoidal category whose objects are
natural numbers with monoidal structure given by the addition, 
and the spaces of morphisms are
\[
\underline{P(\O)}(m,n): = 
\bigoplus_{i_1+\cdots+i_n = m}\O(i_1)\otimes\cdots \otimes \O(i_n),\
\mbox { for $m,n \geq 0$.}
\]  
If $\O$ is multiplicative with multiplication $\mu \in \O(2)$ and unit
${e} \in \O(0)$,  
$P(\O)$ has a distinguished unital monoid
$1 = (1,\mu,{e})$ such that
\begin{equation}
\label{Ani_nezavola.}
\End_{(P(\O),1)} \cong \omu
\end{equation}
in $\MultSO$.
It is now straightforward to verify the rest of the lemma.
\end{proof}

\begin{proof}[Proof of Proposition~\ref{Pisu_v_Koline}]
We construct mutually inverse morphisms $i : \NAT \to \Nat$ and $r :
\Nat \to \NAT$.  For $\alpha \in \NAT^l_{\Rada k1n}$ and $\omu \in
\MultSO$, let $i(\alpha)_{\omu} := \alpha_{(P(\O),1)}$.  We need
to show that the system $i(\alpha)_{\omu}$ is natural in the sense that
diagram~(\ref{Kolecko}) with $i(\alpha)_{\omu} =
\alpha_{(P(\O),1)}$ in place of $\beta_{\omu}$ and
$i(\alpha)_{\pnu} = \alpha_{(P(\P),1)}$ 
in place of $\beta_{\pnu}$, commutes for each map
\begin{equation}
\label{straslive_mlaceni}
\varphi: 
\omu \cong \End_{(P(\O),1)} \to \End_{(P(\P),1)} \cong \pnu
\end{equation}
of multiplicative operads with unit, with the isomorphisms given
by~(\ref{Ani_nezavola.}).  It follows from
naturality~(\ref{Za_chvili_na_bruslicky.})  of $\alpha$
that such diagram commutes if $\varphi$ is induced by a functor $F
: (P(\O),1) \to (P(\P),1)$. Lemma~\ref{Smutek} however shows that all
maps $\varphi$ in~(\ref{straslive_mlaceni}) are induced in this way,
so $i(\alpha):= \{i(\alpha)_{\omu}\}$ is a well-defined family in
$\Nat^l_{\Rada k1n}$. The assignment $\alpha \mapsto i(\alpha)$
obviously extends into an operad map $i : \NAT \to \Nat$.

The inverse map $r$ is constructed as follows.  
For $\beta \in \Nat^l_{\Rada k1n}$ and $(V,M) \in \SymCat$, 
let $r(\beta)_{(V,M)} := \beta_{\End_{(V,M)}}$.  The naturality of
the family $r(\beta)_{(V,M)}$ is clear in this case, so $r(\beta)_{(V,M)} \in
\NAT^l_{\Rada k1n}$. The correspondence $\beta \mapsto r(\beta) :=
\{r(\beta)_{(V,M)}\}$ 
defines a map $r : \Nat \to \NAT$ of $\bbN$-coloured operads.

By definition and~(\ref{Ani_nezavola.}), 
one has $i(r(\beta))_{\omu} = r(\beta)_{(P(\O),1)} =
\beta_{\End_{(P(\O),1)}} = \beta_{\omu}$, so $ir = \id_\Nat$. 
For the opposite composition, 
one gets $r(i(\alpha))_{(V,M)} = i(\alpha)_{\End_{(V,M)}}
  = \alpha_{(P(\End_{(V,M)}),1)}$.
To show that $ri = \id_{\NAT}$, i.e.~that
\begin{equation}
\label{aa}
\alpha_{(P(\End_{(V,M)}),1)} = \alpha_{(V,M)}\
\mbox { for each $(V,M) \in \SymCat$,}
\end{equation}
one uses naturality with respect to the counit 
 $U : (P(\End_{(V,M)}),1) \to (V,M)$.
By  Lemma~\ref{Smutek} again the following diagram commutes
\[
\SquarE
{\underline{V}(M^{\ot k_1},M)\otimes \cdots \otimes
  \underline{V}(M^{\ot k_n},M)}{\underline{V}(M^{\ot l},M)}
{\underline{V}(M^{\ot k_1},M)\otimes \cdots \otimes
  \underline{V}(M^{\ot k_n},M)}{\hphantom{.}\underline{V}(M^{\ot l},M).}
{\alpha_{(P(\End_{(V,M)}),1)}}{\id \ot \cdots \ot \id}{\id}{\alpha_{(V,M)}}
\raisebox{-4em}{\rule{0pt}{0pt}}
\]
This proves~(\ref{aa}) and establishes $ri = \id_\NAT$ as well.
\end{proof}

The last step in the proof of Theorem~A is:

\begin{proposition}
\label{Jarka_si_zapomela_ruzicku.}
The $\bbN$-coloured operads $\Nat$ and $\B$ are naturally isomorphic.
\end{proposition}

\begin{proof}
Denote, as in~(\ref{Zitra_je_pohreb_Jerusciny_maminky.}), by $\Frmu$
be the free multiplicative operad with unit generated by $n$
operations $o_i \in \Fr(k_i)$ of arities $k_i$, $1 \leq i \leq n$.
Let $\omu$ be an arbitrary multiplicative operad with unit. Given
elements $x_i \in \O(k_i)$, $1 \leq i \leq n$, there exists a unique
map $F^{\omu}_{\Rada x1n} : \Frmu \to \omu$ of multiplicative operads
with unit specified by $F^\omu_{\Rada x1n}(o_i) := x_i$, $1 \leq i
\leq n$. Observe that
\begin{equation}
\label{unaven}
F^{\Frmu}_{\Rada o1n} = \id_{\Fr}
\end{equation}
and that, for each map $\varphi : \omu \to \pnu$ of
multiplicative operads with unit, one has
\begin{equation}
\label{Jarunecka}
F^\omu_{\Rada x1n} = \varphi \circ    F^{\pnu}_{\rada{\varphi(x_1)}{\varphi(x_n)}}.
\end{equation}

For a natural operation $\beta \in \Nat^l_{\Rada k1n}$ and $x_i \in
\O(k_i)$ $1 \leq i \leq n$ as above, one has, by
naturality~(\ref{Kolecko}), the commutative diagram
\begin{equation}
\label{JarUska}
\SquaRE{\Fr(k_1) \ot \cdots \ot \Fr(k_n)}{\Fr(l)}
{\O(k_1)\ot \cdots \ot \O(k_n)}{\hphantom{.}\O(l).}
{\beta_{\Frmu}}{F^\omu_{\Rada x1n}(k_1) \ot 
\cdots \ot  F^\omu_{\Rada x1n}(k_n)}
{F^\omu_{\Rada x1n}(l)}{\beta_{\omu}}
\raisebox{-4em}{\rule{0pt}{0pt}}
\end{equation}
We will need also a particular case of~(\ref{JarUska}) when $\O = \Fr$
and $x_i = u_i o_i$, for some scalars $u_i \in \bbZ$, $1 \leq i \leq n$:
\begin{equation}
\label{JarUskA}
\SquaRE{\Fr(k_1) \ot \cdots \ot \Fr(k_n)}{\Fr(l)}
{\Fr(k_1) \ot \cdots \ot \Fr(k_n)}{\hphantom{.}\Fr(l).}
{\beta_{\Frmu}}{F^{\Frmu}_{u_1o_1,\ldots,u_no_n}(k_1) \ot \cdots \ot
 F^{\Frmu}_{u_1o_1,\ldots,u_no_n}(k_n)} 
{F^{\Frmu}_{u_1o_1,\ldots,u_no_n}(l)}{\beta_{\Frmu}}
\raisebox{-4em}{\rule{0pt}{0pt}}
\end{equation}

Recall that $\Fr_{\rada 11}(l)$ denotes the subspace of
elements of $\Fr(l)$ containing each generator $\Rada o1n$ 
precisely once. We begin the proof by observing that each natural operation
$\beta \in \Nat^l_{\Rada k1n}$ determines an element $\xi(\beta) \in
\Fr(l)$ by
\begin{equation}
\label{=}
\xi(\beta) := \beta_{\Frmu}(o_1 \otimes \cdots \ot o_n) \in \Fr(l).
\end{equation}
We need to show that $\xi(\beta)$ in
fact belongs to $\Fr_{\rada 11}(l)  \cong \B^l_{\Rada k1n}$.  
Clearly, $\Fr(l)$ decomposes as
\[
\Fr(l) = \textstyle\bigoplus_{\Rada j1n \geq 0} \Fr_{\Rada j1n}(l), 
\]
where $\Fr_{\Rada j1n}(l) \subset \Fr(l)$ is the subspace of elements
having precisely $j_i$ instances of $o_i$ for each $1 \leq i \leq
n$. The endomorphism $F^\fmu_\pom : \Frmu \to \Frmu$ 
acts on $\Fr_{\Rada j1n}(l)$ 
by the multiplication with $u_1^{j_1}
\cdots u_n^{j_n}$ --  $\Fr_{\Rada j1n}(l) \subset \Fr(l)$
is, in fact, characterized by this property.  The element $\xi(\beta)$
uniquely decomposes as $\xi(\beta) = \sum_{\Rada j1n \geq 0}
\xi(\beta)_{\Rada j1n}$, for some $\xi(\beta)_{\Rada j1n} \in \Fr_{\Rada
j1n}(m)$.

Let us turn our attention to diagram~(\ref{JarUskA}). 
By the definition of the map $\F^\fmu_\pom$, one has
\[
\beta_{\Frmu}(F^\fmu_\pom(k_1) \ot \cdots \ot
  F^\fmu_\pom (k_n))(o_1 \ot \cdots \ot o_n)
  =\beta_{\Frmu}(u_1o_1 \ot \cdots \ot u_no_n),
\]
while the linearity of $\beta_{\Frmu}$ implies
\[
\beta_{\Frmu}(u_1o_1 \ot \cdots \ot u_no_n) =
u_1\cdots u_n \cdot \beta_{\Frmu} (o_1 \ot \cdots \ot o_n) = u_1\cdots
u_n \cdot \xi(\beta).
\]
On the other hand
\[
F^\fmu_\pom(l)(\beta_{\Frmu})(o_1 \ot \cdots \ot o_n) = 
F^\fmu_\pom(l)(\xi(\beta)) =  
\sum_{\Rada j1n \geq 0}  u_1^{j_1}\cdots u_n^{j_n}  
\cdot \xi(\beta)_{\Rada j1n},
\]
therefore the commutativity of~(\ref{JarUskA}) means that 
\[
u_1 \cdots u_n \cdot \xi(\beta) =
\sum_{\Rada j1n \geq 0}   u_1^{j_1}\cdots u_n^{j_n}  
\cdot \xi(\beta)_{\Rada j1n}
\]
for each $\Rada u1n \in \bbZ$. We conclude that $\xi(\beta)_{\Rada j1n} = 0$ if
$(\Rada j1n) \not= (\rada 11)$, so $\xi(\beta) = \xi(\beta)_{\rada 11}
\in \Fr_{\rada 11}(l)$.

Let us prove that, vice versa, each element $\eta \in \Fr_{\rada 11}(l)
\cong \B^l_{\Rada k1n}$ determines a natural operation $\gamma(\eta) \in
\Nat^l_{\Rada k1n}$ by the formula 
\begin{equation}
\label{Jarunka-pusa}
\gamma(\eta)_{\omu}(x_1 \ot \cdots \ot  x_n) = 
F^\omu_{\Rada x1n}(l)(\eta),\ \mbox { for each $x_i \in \O(k_i)$,
$1 \leq i \leq n$.}
\end{equation}
The linearity of the map
$\gamma(\eta)_{\omu} : \O(k_1) \ot \cdots \O(k_n) \to \O(l)$
defined in this way easily follows from the fact that $\eta$ belongs
to the subspace $\Fr_{\rada 11}(l) \subset \Fr(l)$. The commutativity
of diagram~(\ref{Kolecko}) for the family $\gamma(\eta) =
\{\gamma(\eta)\}_{\omu}$ defined by~(\ref{Jarunka-pusa}) is
equivalent to
\[
\varphi (\gamma(\eta)_{\omu})(x_1 \ot \cdots \ot  x_n) =
\gamma(\eta)_{\pnu}(\varphi (k_1) \ot \cdots \ot
\varphi(k_n))(x_1 \ot \cdots \ot  x_n) 
\] 
for each $x_i \in \O(k_i)$, $\mezi in$. This, by the
definition~(\ref{Jarunka-pusa}) of $\gamma(\eta)$, means
\[
F^\omu_{\Rada x1n}(l)(\eta) = 
\varphi \circ F^{\pnu}_{\rada{\varphi(x_1)}{\varphi(x_n)}}(l)(\eta)
\]
which follows from~(\ref{Jarunecka}).
So $\gamma(\eta) = \{\gamma(\eta)_{\omu}\}$ is a well-defined
natural operation in $\Nat^l_{\Rada k1n}$.

Let us prove that the maps $\xi : \Nat^l_{\Rada k1n} \to \Fr_{\rada
11}(l)$ and $\gamma :  \Fr_{\rada 11}(l)  \to  \Nat^l_{\Rada k1n}$ defined
by the above correspondences are mutual inverses. We have, for $\eta
\in \Fr_{\rada 11}(l)$, 
\[
\xi(\gamma(\eta)) = \gamma(\eta)_\fmu(o_1 \ot \cdots \ot o_n) = 
F^\fmu_{\Rada o1n}(l)(\eta) = \eta,
\]
where the last equality follows from~(\ref{unaven}). For the opposite
composition, one has
\begin{eqnarray*}
\gamma(\xi(\beta))_\omu(x_1 \ot \cdots \ot x_n) &=&
F^\omu_{\Rada x1n}(l)(\xi(\beta)) =
F^\omu_{\Rada x1n}(l)(\beta_\fmu)(o_1 \ot \cdots o_n)
\\
&=&
\beta_\omu(F^\omu_{\Rada x1n}(k_1)(o_1) \ot \cdots \ot 
F^\omu_{\Rada x1n}(k_n)(o_n))
\\
&=&
\beta_\omu(x_1 \ot \cdots \ot x_n),
\end{eqnarray*}
where we used the commutativity of~(\ref{JarUska}) and the definition of
$F^\omu_{\Rada x1n}$. The proof is finished by observing that the two
mutually inverse families of maps  $\xi : \Nat^l_{\Rada k1n} \to \Fr_{\rada
11}(l) \cong \B^l_{\Rada k1n}$ and  $\gamma :  
\B^l_{\Rada k1n} \cong \Fr_{\rada 11}(l)  \to \Nat^l_{\Rada k1n}$
define two mutually inverse homomorphisms $\xi : \Nat \to \B$ and
$\gamma : \B \to \Nat$ of $\bbN$-coloured operads.
\end{proof}

We  have a nonsymmetric
variant of Proposition~\ref{na_letisti}.

\begin{proposition}
\label{dnes_Jarca_na_houbicky}
The category of algebras over the $\bbN$-coloured operad $T$ is
isomorphic to the category of multiplicative nonsymmetric operads with
unit.
\end{proposition}

\section{Crossed interval groups}
\label{2_den}

In this section we introduce crossed interval groups. We believe that
this new concept is important on its own and we therefore decided to
include also some results and examples not directly related to our
immediate applications. The reader who is not interested in this more
general picture may read only the material necessary for this paper --
Definition~\ref{ci}, Example~\ref{IS}, Proposition~\ref{Jaruska},
Definition~\ref{ttt}, Lemma~\ref{sss}, Lemma~\ref{freeHobject} -- and
skip the rest.

 {\em The category of intervals\/} $\I$ is the category whose objects
are the linearly ordered sets $\int n := \{-1,\rada 0n,n+1\}$, $n \geq
-1$, and morphisms non-decreasing maps $f :\int m \to \int n$ such
that $f(-1) = -1$ and $f(m+1) = n+1$.  The points $-1,n+1$ are called
the {\em boundary points\/} of $\int n$ and the ordinal $[n]:=
\{0,\ldots,n\}$ is the {\em interior\/} of $\int n$. Intuitively, one
may think of $\int n$ as of the interval with endpoints $-1$ and
$n+1$; morphisms in $\I$ are then non-decreasing, boundary preserving
maps of intervals.  The following statement goes back, probably, to
Gabriel and Zisman \cite[Section~III.1.1]{GZ} but became known
under the name Joyal's duality due to an influential preprint
\cite{joyal:disks}. The reader who is interested to understand the
details of Joyal's duality including higher dimensional generalization
is referred to \cite{MZ}.

\begin{proposition}[Joyal's duality]  The category 
$\I$ is isomorphic to the opposite
category $\Delta^\op$ of the (skeletal) category $\Delta$ of  nonempty finite
ordered sets and their non-decreasing maps.
\end{proposition} 

\begin{proof}
To construct the isomorphism $\joy :
\Delta^\op \stackrel\cong\to \I$,
recall that $\Delta$ has objects the finite ordinals $[n] = \{\rada
0n\}$, $n \geq 0$.  Morphism of $\Delta$ are generated by the face
maps $\delta_i : [n-1] \to [n]$ (misses $i$) and the degeneracy maps
$\sigma_i : [n+1] \to [n]$ (hits $i$ twice), $i = \rada0n$. Then $\joy
([n]) := \int {n-1}$, for $n \geq 0$. Moreover, $d_i=\joy(\delta_i) :
\int {n-1} \to \int {n-2}$ is the map that hits $i-1$ twice, and $s_i=
\joy(\sigma_i): \int {n-1} \to \int {n}$ the map that misses $i$. 
\end{proof}

The category $\I$ contains a subcategory of open maps. The objects of
this subcategory are intervals and the morphisms are interval
morphisms $f : \int m \to \int n$ which preserve the interiors i.e.
$f(\{\rada 0m\}) \subset \{\rada 0n\}$. This subcategory is isomorphic
to the (skeletal) category of all finite ordinals called
$\Delta_{alg}$ (algebraic $\Delta$). The category $\Delta_{alg}$ has a
well-known universal property -- it is a monoidal category freely
generated by a monoid.  The following definition is motivated
by~\cite[Definition~1.1]{fiedorowicz-loday:TAMS91}.

\begin{definition}
\label{ci}
A {\em crossed interval group\/} $\H$ is a sequence of groups
$\{H^n\}_{n \geq 0}$ equipped with the following structure:
\begin{enumerate}
\item[(a)]
a small category $\I H$ called the {\em associated category\/} such that
the objects of $\I H$ are the ordinals $\int n = \{-1,\rada 0n,n+1\}$, 
$n \geq -1$,
\item[(b)]
$\I H$ contains $\I$ as a subcategory,
\item[(c)]
$\Aut_{\I H}(\int n) = {(H^{n+1})}^{\op}$, $n \geq -1$, and
\item[(d)] any morphism $f: \int m \to \int n$ in $\I H$ can be
uniquely written as $\phi \circ h$, where $\phi \in \Hom_\I(\int
m,\int n)$ and $h \in {(H^{m+1})}^\op$.
\end{enumerate}
\end{definition}

We would like to extend the Joyal's duality to the case of crossed
interval groups. To this end, we need

 \begin{definition}
\label{cc} 
A {\em crossed cosimplicial group\/} $\H$ is  
a sequence of groups $\{H^n\}_{n \geq 0}$ equipped with the following structure:   
\begin{enumerate}
\item[(a)]a small category $H\Delta$ such that
the objects of $H\Delta$ are the ordinals $[n] = \{0,\rada ,n\}$, 
$n \geq 0$,
\item[(b)]
$H\Delta$ contains $\Delta$ as a subcategory,
\item[(c)]
$\Aut_{H\Delta}([ n]) = H^n$, $n \geq 0$, and
\item[(d)] any morphism $f: [ m] \to [ n]$ in $H\Delta$ can be
uniquely written as $h\circ\phi  $, where $\phi \in \Hom_\Delta(
[m], [n])$ and $h \in H^n$.
\end{enumerate}
\end{definition}

Now, given a crossed interval group $\H$ we construct its {\it Joyal's
dual\/} crossed cosimplicial group $J\H$ with the associated category
$(JH)\Delta$  by taking the same sequence of
groups $\{H^n\}_{n\ge 0}$ and putting $(JH)\Delta ([n],[m]) := \I H
(\int {m-1},\int {n-1})$. One can easily check using the ordinary
Joyal's duality that we thus obtain a crossed cosimplicial
group. Moreover, we also obtain an isomorphism of the associated
categories $j_H:((JH)\Delta)^{\op} \rightarrow \I H .$ We therefore
see that crossed interval groups and crossed cosimplicial groups
are equivalent notions.  

An important crossed interval group is introduced in the following
example.

\begin{example}
\label{IS}
Let $\I S$ be the category whose objects are the ordinals 
$\int n = \{-1,0,\ldots,n,n+1\}$, 
$n \geq -1$, and morphisms are maps $f : \int m \to \int n$ such that
\begin{enumerate}
\item[(a)]
a linear order on each non-empty fiber $f^{-1}(i)$, $i \in \int n$ is
specified,
\item[(b)]
$f$ preserves the endpoints and their order, 
that is  $f(-1) = -1$ and $f(m+1) = n+1$, and
\item[(c)] 
$-1$ is the minimal element in $f^{-1}(-1)$ and $m+1$ is
the maximal element in $f^{-1}(n+1)$.
\end{enumerate}
\end{example}
The group $\Aut_{\I S}(\int n)$ equals the symmetric group $S_{n+1}$,
whence the notation. We leave as an exercise to verify that the $\I
S$ defined above is an associated category of a crossed interval
group. We denote this crossed interval group $\S$.

Let $\Delta S$ be the category whose objects are the ordinals $[n] :=
\{\rada 0n\}$ and morphisms are maps $g : [m] \to [n]$ with a
specified linear order on the fibers $g^{-1}(i)$, $i \in [n]$. Recall
that $\Delta S$, called the {\em symmetric category\/}, is an
associated category of a {\em crossed simplicial group\/} $S_*$ in the
sense of~\cite[Definition~1.1]{fiedorowicz-loday:TAMS91}.  There is
also an `algebraic' version of the symmetric category $(\Delta
S)_{alg}$\label{moz} 
introduced by Day and Street in \cite{DayStreet:substitution}, which has
a universal property similar to $\Delta_{alg}$: the category $(\Delta
S)_{alg}$ is the symmetric monoidal category freely generated by a
(noncommutative) monoid.

Analogously to the inclusion $\Delta_{alg}\hookrightarrow \I $, one has
the inclusion
\[
I : (\Delta S)_{alg} \hookrightarrow \IS
\]
such that $I([n]) := \int n$. If $g : [m] \to [n]$ is a morphism in $
(\Delta S)_{alg}$, then $I(g)$ is the unique morphism $f : \int m \to
\int n$ such that $f|_{\{\rada 0m\}} = g$. The image of $I$ is the
subcategory of {\em open maps\/}, i.e.~maps $f : \int m \to \int n$
such that $f(\{\rada 0m\}) \subset \{\rada 0n\}$.  Likewise, there is
the inclusion
\begin{equation}
\label{Napise_Magda_jeste?}
\iota : \IS \hookrightarrow \Delta S
\end{equation}
such that $\iota(\int n) := [n+2]$ and, for $f : \int m \to \int n$,
$g =\iota(f)$ is the map with $g(i) := f(i-1)$, for $i \in [n+2]$.  

\begin{example}
\label{Z2}
Another, rather trivial example of a crossed interval group is the
 `constant' group $\Z_2.$ The objects are the intervals $\int n$, $n
 \geq -1 ,$ and the morphisms are maps $f : \int m \to \int n$ such
 that $f$ restricted to the interior of the interval is order
 preserving and $f$
preserves the boundary, that is, $f(\{-1,m+1\}) = \{-1,n+1\}.$ It is not hard to see that all the automorphisms groups
 in this category are $\bbZ_2$.

\end{example}

\begin{example}
\label{Flip}
The crossed interval group $\Flip$. The objects are again the
intervals. The morphisms from $\int n$ to $\int m$ are maps $f : \int
m \to \int n$ that preserve the boundary, which are either order
preserving or order reversing.  Again, it is not difficult to see that
all automorphism groups are $\bbZ_2$. The unique nontrivial automorphism
of an interval is the flip around its center. The crossed interval
groups $\Z_2$ and $\Flip$ are, however, not isomorphic.
\end{example}

\begin{example} 
One can construct an example of a crossed interval group 
${\mathcal B}{\it raid}$ 
if one replaces
symmetric groups by braid groups in Example~\ref{IS}. The braid
analogue of Example~\ref{Z2} will be the `constant' crossed interval
group $\Z .$ We do not know if there is a braid analogue of the
crossed interval group $\Flip$ of Example~\ref{Flip}.

\end{example}  

\begin{remark}The examples of $\Z_2$ and $\Flip$ indicate that the classification of the
crossed interval groups is, probably, more subtle than the
classification of crossed simplicial groups given in
\cite{fiedorowicz-loday:TAMS91}.
\end{remark}
One advantage of the crossed interval groups over crossed cosimplicial groups  is the same variance
of morphisms as in Loday-Fiedorowicz crossed simplicial groups. So,
many basic results for crossed simplicial groups can be carried over
to crossed interval groups without changing the formulas from
\cite{fiedorowicz-loday:TAMS91}. For example, we have that for any
$h\in H^n$ and any $\phi\in \I (\int m , \int n)$ there exist unique
$\phi^*(h)\in H^m$ and $h^*(\phi) \in \I (\int m ,\int n )$ such that
$h\circ \phi = h^*(\phi) \circ \phi^*(h) .$ Using these
correspondences one can give an alternative characterization of
crossed interval groups repeating
\cite[Proposition~1.6]{fiedorowicz-loday:TAMS91} verbatim.

Let us define an {\em automorphism of an interval\/} $\int n$ as a
bijection $\int n$ to $\int n$ which preserves the set of boundary
points.  Clearly, the group $\Aut\int n $ of all automorphisms is
isomorphic to $S_{n+1}\times \bbZ_2$, where $S_{n+1}$ denotes the
symmetric group on $n+1$ elements and $\bbZ_2 := \bbZ/2\bbZ$.

The following characterization of the crossed interval groups is
slightly more complicated than its crossed simplicial analogue (see
Proposition 1.7 from~\cite{fiedorowicz-loday:TAMS91}).

\begin{proposition}
\label{Jaruska}
A crossed interval group $\H$ is a cosimplicial set $H^\bullet =
\{H^n\}_{n \geq 0}$ such that each $H^n$ is a group, together with a group
homomorphism
\begin{equation}
\label{Jar}
\rho_n: H^n \rightarrow  \Aut \int {n-1} \cong  S_n \times \bbZ_2
\end{equation}
for each $n \geq 0$, such that, for $h,h' \in H^n$,
\begin{enumerate} 
\item 
$d_i(hh') = d_{\overline{h}(i)}(h')d_i(h)$, 
where $\overline{h}(i) := \rho_n(h)(i-1) +1$, $0 \leq i \leq n+1$, 
\item 
$s_i(hh') = s_{\underline{h}(i)}(h')s_i(h)$, 
where $\underline{h}(i) := \rho_n(h)(i)$,  $0\le i \le n-1$.  
\item 
Moreover, the following set diagrams are commutative:
\[
\Square 
{\int n}{\int {n-1}}{\int n}{\int {n-1}}%
{d_{\overline{h}(i)}}{d_i(h)}{h}{d_i}\hspace{30mm}
\Square 
{\int {n-2}}{\int {n-1}}{\int {n-2}}{\int {n-1}}%
{s_{\underline{h}(i)}}{s_i(h)}{h}{s_i}
\]
where $0 \leq i \leq n+1$ in the first and   $0 \leq i \leq n-1$ in
the second diagram.
\end{enumerate}
\end{proposition}

The proof is basically identical to the proof of Proposition~1.7
in \cite{fiedorowicz-loday:TAMS91} but we have to take into account
that among the coface operators $d_i$ in $\I$ the first and the last
operators are not open maps of intervals. One can check that under the
correspondence $h^* : \I (\int m , \int n) \to \I (\int m , \int n)$ 
the open maps go to the open maps, so the first
and the last operators can only be stable or permuted by $h^*$. This
explains the factor $\bbZ_2$ in~(\ref{Jar}).

We also observe that $\rho_*$ is a homomorphism of crossed interval
groups $ \H \rightarrow \S \times \Z_2$, so
the value of $\rho_n$ is determined by its action on the set of open
coface operators $\{d_i;\ 1\le i \le n\}$ or equally on the set of
codegenaracies (which are automatically open maps of intervals) $\{s_i
;\ 0\le i \le n-1\}$, and the homomorphism $\rho_0:H^0 \rightarrow \bbZ_2$.

\begin{definition} 
\label{ttt}
Let $\H$ be a crossed interval group and let $\C$ be a category. An
{\em $\H$-object\/} in $\C$ is a functor $X:(\I H)^{\op}\rightarrow \C .$ A
{\em morphism\/} between two $\H$-objects is a natural transformation
of functors.
\end{definition}  

Similarly to the crossed simplicial case we have the following simple
characterization of $\H$-objects (compare with Lemma 4.2
of~\cite{fiedorowicz-loday:TAMS91}). We formulate it, as well as
Lemma~\ref{freeHobject} below, in the more intuitive language of 
elements, though, of
course, objects of a general cocomplete category do not have `elements.' The
concerned reader may easily translate the statements into the formally
correct language of arrows. 

\begin{lemma}
\label{sss}
The category of $\H$-objects in a cocomplete  category $\C$ is equivalent to the category whose
objects are cosimplicial objects $X^\bullet$ in $\C$ equipped with the
following additional structure:
\begin{enumerate} 
\item right group actions
$X^n\times H^n \rightarrow X^n$, $n \geq 0$,
\item coface relations 
$d_i(xh) = d_{\overline{h}(i)}(x)d_i(h)$, for
$x \in X^n$, $h \in H^n$, $0 \leq i \leq n+1$,   
\item codegeneracy relations $s_i(xh) = s_{\underline{h}(i)}(x)s_i(h)$,
for $x \in X^n$, $h \in H^n$, $0 \leq i \leq n-1$, 
\end{enumerate}
and whose morphisms are cosimplicial morphisms which are degreewise
equivariant. 
\end{lemma}

In the above lemma, $\overline{h}(i)$ and~$\underline{h}(i)$ have the
same meaning as in Proposition~\ref{Jaruska}.  There is a~natural
restriction functor $\U_{\H}$ from the category of $\H$-objects in $\C$ to the
category of cosimplicial objects in $\C.$ This functor has a left
(right) adjoint provided $\C$ is cocomplete (complete). We will denote
the left adjoint by $\F_{\H} .$ It is easy to get explicit formulas
for $\F_{\H}$ similar to the formulas
from~\cite[Definition~4.3]{fiedorowicz-loday:TAMS91}:

\begin{lemma}
\label{freeHobject} 
Let $X= X^\bullet$ be a cosimplicial object in a 
cocomplete category $\C .$ Then $\F_{\H}(X)$ is an $\H$-object with
$\F_{\H}(X)^n := X^n\times H^n$, $n \geq 0$, and with cofaces and
codegeneracies given by
\begin{eqnarray*}
d_i(x,h) := (d_{\overline{h}(i)}(x),d_i(h)),\ 
\mbox{for $x \in X^n$, $h \in H^n$, $0 \leq i \leq n+1$},
\\        
s_i(x,h) := (s_{\underline{h}(i)}(x),s_i(h)),\ 
\mbox{for $x \in X^n$, $h \in H^n$, $0 \leq i \leq n-1$}.
\end{eqnarray*}
The unit $\iota: X \mapsto F_{\H}(X)$ of the monad $F_{\H}$ generated
by the adjunction $\F_{\H}\dashv \U_{\H}$ is given by the cosimplicial map
$\iota(x) := (x,1) .$
\end{lemma}

\section{A homotopy equivalence of cosimplicial totalizations.}
\label{jarKA}

\def\S{\EuScript S}  \def\id{{\it id}}
\def\Ab{\EuScript Ab} \def\bbZ{{\mathbb Z}} \def\bbN{{\mathbb N}}
\def\frakM{{\mathfrak M}} \def\A{\EuScript A}
\def\hatF{{\widetilde F}} \def\sgn{{\rm sgn}}
\def\hatH{{\widetilde H}} \def\Tot{{\rm Tot\/}}
\def\hatB{{\widetilde B}}
\def\Deltaop{{\Delta^{{\it op}}}}
\def\Coch{\EuScript Coch}\def\CochChain{\EuScript {C}och\EuScript{C}hain}
\def\Ch{\EuScript Ch} \def\uCh{\underline{\EuScript Ch}}
\def\Chain{\EuScript Chain}
\def\C{\EuScript C} \def\Im{{\it Im\/}} \def\Ker{{\it Ker\/}}
\def\Set{\EuScript Set}
\def\bsigma{{\overline{\sigma}}}
\def\Rada#1#2#3{{#1_{#2},\ldots,#1_{#3}}}
\def\unit{{1 \!\! 1}} \def\whD{{\widehat D}}

This section is devoted to a result about cosimplicial totalizations of
$\S$-chain complexes, where $\S$ is the crossed interval group
introduced in Example~\ref{IS}. We recall 
first the notion of cosimplicial
totalization that associates to a cosimplicial chain complex
$C^{\bullet}_*$ a chain complex $\uTot(C^{\bullet}_*).$ The main result of
this section, Theorem~\ref{t2}, states that the cochain complex
associated to $F_\S(C^\bullet_*)$, where $C^\bullet_*$ is a cosimplicial
chain complex with finitely generated torsion-free components, is
canonically homotopy equivalent to the cochain complex associated to
$C^\bullet_*$. Our main technical tool will be a cochain version of the
method of acyclic models developed in Appendix~\ref{sec:acycl-models-theor}.

This result should be considered as dual to the result about crossed
simplicial group $\Delta S$ which says that the geometric realization
of simplicial topological space $F_{\Delta S}(X_{\bullet})$ is
homotopy equivalent to the geometric realization of $X_{\bullet}.$ The
proof of this statement is relatively easy.  Theorem 5.3 from
\cite{fiedorowicz-loday:TAMS91} provides an equivariant homeomorphism
of geometric realizations
\[
|F_G(X_\bullet)| \rightarrow |G_\bullet|\times |X_\bullet |
\] 
for any crossed simplicial group $G.$
Then it is enough to observe  that  the geometric realization of the simplicial set $\Delta S_\bullet $ is
contractible by explicitly constructing a simplicial contraction.

Unfortunately, the methods of \cite{fiedorowicz-loday:TAMS91} do not
work for the crossed interval groups.  So we need to develop a
different approach.  Our argument is based on a corollary (Proposition
\ref{p1}) to a version of Acyclic Model Theorem for cochain functors
which we formulate in the Appendix~\ref{sec:acycl-models-theor}.

Let  $\Chain$ be
the category   of $\bbZ$-graded chain
complexes of abelian groups. It is  a closed symmetric monoidal
category with respect to the usual tensor product of chain complexes.
Let $\Chain^\Delta$ be the category of cosimplicial objects in
$\Chain$.  Objects of $\Chain^\Delta$ are systems $C^\bullet_* =
(C^n_*,\pa)$, $n \geq 0$, of chain complexes with face and degeneracy
operators $d_i : C^{n-1}_* \to C^n_*$, $s_i: C^{n+1}_* \to C^n_*$, $0
\leq i \leq n$, which are chain maps and satisfy the standard
cosimplicial identities.

Let  $\CochChain$  be the category of
non-negatively graded cochain complexes in the category $\Chain$.
Objects of $\CochChain$ are systems $(E^*_*,d,\pa)$ consisting of
abelian groups $E^m_k$, $m \geq 0$, $k \in \bbZ$, and linear maps $d:
E^m_k \to E^{m+1}_k$ ({\em horizontal\/} differentials), $\pa :
E^m_{k+1} \to E^m_k$ ({\em vertical\/} differentials) that square to
zero and satisfy $\pa d + d \pa = 0$.

Each cosimplicial chain complex $C^\bullet_* \in \Chain^\Delta$ 
determines an object $C^*_*$ of
$\CochChain$, with the differential $d : C_k^m
\to C_k^{m+1}$ given by the standard formula
\[
d(c) := \sum_{0 \leq i \leq m+1} (-1)^{i+m} d_i (c),\ c \in C^m_k,\ m
\geq 0,\ k \in \bbZ.
\]
The correspondence $C^\bullet_* \mapsto C^*_*$ defines  
the functor `replace bullet by star' $\bs : \Chain^\Delta
\to \CochChain$.

For $E^*_* \in \CochChain$ denote by $|E^*_*|^*$ the cochain complex
with $|E^*_*|^m := \prod_{n -k = m}E^n_k$, with the differential given by
\[
d(a^0,a^1,a^2,\ldots) := (-\pa a^0,d a^0 - \pa a^1, d a^1 - \pa a^2,\ldots),
\]
for $(a^0,a^1,a^2,\ldots) \in E^0_* \times E^1_* \times E^2_* \times
\cdots$.  Finally, define the {\em total complex\/} of a cosimplicial
chain complex $C^\bullet_*$ to be the cochain complex
$\uTot(C^\bullet_*)^* := |\bs(C^\bullet_*)|^*$. The bar over $\Tot$ indicates
that we totalize with respect to the cosimplicial structure.
So we have three functors forming a commutative triangle:
\newcommand{\VVtriangle}[6]{
\setlength{\unitlength}{2em}
\begin{picture}(5,3.6)(0,-.1)
\thicklines
\put(0,2.5){\makebox(0,0){$#1$}}
\put(5,2.5){\makebox(0,0){$#2$}}
\put(2.5,0){\makebox(0,0){$#3$}}

\put(1,2.5){\vector(1,0){2.6}}
\put(0.5,2){\vector(1,-1){1.5}}
\put(4.5,2){\vector(-1,-1){1.5}}

\put(1,1){\makebox(0,0)[rt]{$#5$}}
\put(2.5,2.8){\makebox(0,0)[b]{$#4$}}
\put(4,1.2){\makebox(0,0)[lt]{$#6$}}
\end{picture}
}
\[
\VVtriangle {\Chain^\Delta}{\CochChain}{\Coch}{\bs}{\uTot}{|-|}
\]

\begin{warning}
Objects of $\CochChain$ are particular types of bicomplexes of
abelian groups (with a slightly unusual degree convention), but our
definition of $|E^*_*|^*$ differs from the traditional definition of the
total complex of a bicomplex in that it involves the direct {\em product\/}
instead of the direct {\em sum\/}.
\end{warning}

\begin{example}
The components of the  operad $\Big$ form a cosimplicial,
non-negatively graded chain complex $B^\bullet_*(n)$ with $B^q_m(n) =
\bigoplus_{k_1+\cdots +k_n=m} B^q_{k_1,\ldots,k_n}$. The corresponding
total space $\uTot(B^\bullet_*(n))$ equals the arity $n$ component
$\Big(n)$ of the dg operad $\Big$. Another example is the cosimplicial
chain complex determined, in the same manner, by the components of the
suboperad~$\Tam \subset \Big$.
\end{example}

To simplify the notation, we will drop, in the rest of this section,
the stars $*$ from the notation whenever the nature of the
corresponding object is clear from the context.
The unit of the monad from  Lemma~\ref{freeHobject} induces the canonical
inclusion of cosimplicial chain complexes
\[
\iota : C^\bullet  \to F_\S(C^\bullet),
\]
given by $\iota(a) := (a,1)$.

\begin{definition}
The {\em miraculous map\/} $m : \bs(F_\S(C^\bullet)) \to
\bs(C^\bullet)$ is defined by the formula
\[
m(a,\sigma) := \sgn(\sigma) a,\ a \in C^q,\ \sigma \in S_q,\ q \geq 0.
\]
\end{definition}

We call $m$ miraculous because it is not induced by a cosimplicial
map. Obviously $m \circ \bs(\iota) = \id$.

\begin{proposition}
\label{p44}
The miraculous map $m$ is a morphism in $\CochChain$.
\end{proposition}

\begin{proof}
Recall (see Lemma~\ref{freeHobject}) 
that the differential acts on $(a,\sigma) \in C^q \times S_q$ by
\[
d^q(a,\sigma) = 
\sum_{0 \leq i \leq q+1} (-1)^i(d_{\bsigma(i)}(a),d_i(\sigma)),
\]
where $\bsigma : \{0,\ldots,q+1\} \to \{0,\ldots,q+1\}$ 
fixes $0,q+1$ and equals $\sigma$
on $\{1,\ldots,q\}$. The coboundary operators $d_i : S_q \to S_{q+1}$
are given by  $d_0(\sigma) := \unit \times \sigma$, $d_{q+1}(\sigma) := \sigma
\times \unit$ and $d_i(\sigma)$ is the permutation
obtained by doubling the $i$th input of $\sigma$. These coboundary operators
were introduced in~\cite[Example~12]{markl:de}.
So, for $(a,\sigma) \in C^q \times S_q = F_\S(C^\bullet)^q$,
\[
m d^q(a,\sigma) = \sum_{0 \leq i \leq q+1} (-1)^i
\sgn(d_i(\sigma)) d_{\bsigma(i)}(a)  
\]
while 
\[
d^q m (a,\sigma) = \sum_{0 \leq i \leq q+1} (-1)^i
\sgn(\sigma) d_i(a) =  \sum_{0 \leq i \leq q+1} (-1)^{\bsigma(i)}
\sgn(\sigma)d_{\bsigma(i)}(a).
\]
Comparing the terms at $d_{\bsigma(i)}(a)$ one sees that 
$m$ will be a cochain map if we prove that
\[
(-1)^{\bsigma(i)} \sgn(\sigma) =  (-1)^i \sgn(d_i(\sigma)), \ \mbox {
for } 0 \leq i \leq q+1.
\]
The last equality can be established by a direct verification.
\end{proof}

\begin{theorem}
\label{t2}
Let $C^\bullet$ be a cosimplicial chain complex with finitely generated
torsion-free components. Then 
\[
\uTot (\iota) : \uTot(C^\bullet) \to \uTot(F_\S(C^\bullet)) 
\mbox { and }
|m| : \uTot(F_\S(C^\bullet)) \to \uTot(C^\bullet)
\]
are natural mutually inverse homotopy equivalences. Moreover, 
$|m| \circ \uTot(\iota) = \id$ and there is a natural cochain homotopy between 
$\uTot(\iota) \circ |m|$ and
$\id$.
\end{theorem}

\begin{proof}

The equation $|m| \circ \uTot(\iota) = \id$ is trivial. Therefore it
remains to prove the existence of a natural cochain homotopy between
$\uTot(\iota) \circ |m|$ and $\id.$ We do it by applying
Proposition~\ref{p1} of Appendix~\ref{sec:acycl-models-theor}.

For $A \in \Chain$, let $A^\# := \uChain(A,\bbZ)$, where $\bbZ$
are the integers considered as the chain complex concentrated in
degree~$0$, denote the linear dual.  The components $(A^\#)_n$ of the dual
$A^\#$ are given by
\begin{equation}
\label{JARkA}
(A^\#)_n = A_{-n}^\#,\ n \in \bbZ.
\end{equation}

For $q \geq 0$ we
denote by $\Delta[q]$ the standard simplicial $q$-simplex and $M[q]$ the
simplicial abelian group spanned by $\Delta[q]$. Let $D[q]$ be the
cosimplicial abelian group given as the componentwise linear dual of
$M[q].$  We will interpret $M[q]$ (resp.~$D[q]$) 
as a simplicial (resp.~cosimplicial) chain complex concentrated in
degree zero, with trivial vertical differentials.

Let $\Ch \subset \Chain$ be the  subcategory consisting of
chain complexes with finitely generated torsion-free components. 
To apply Proposition \ref{p1}
we put $\D = \Ch^\Delta$, $A = B = \bs(F_\S(-))$, $f = \bs(\iota)
\circ m$, $g =\id$ and $\frakM = \{D[0],D[1],D[2],\ldots\}$. We prove
that $F_\S(-)^q$ is corepresented, with models $\frakM$, in
Lemma~\ref{corep} below.  It is also obvious that $(\bs(\iota) \circ
m)^0 = \id^0$ and that $F_\S(D[q])^\bullet$ is concentrated in chain
degree $0$.  The acyclicity of models is established in Proposition
\ref{p22}. Hence, all the conditions of Proposition~\ref{p1} are
satisfied and we complete the proof of the theorem.
\end{proof}

\begin{lemma}\label{corep} The functor 
$F : \Ch^\Delta \to \Chain$ defined by
$F(C^\bullet_*) := F_\S(C^\bullet_*)^q$ 
is corepresented with the model $D[q]$.
\end{lemma}
\begin{proof} Observe that the functor $F_\S:\Chain^\Delta\to\Chain^\Delta$ takes objects from $\Ch^\Delta$ to the objects of $\Ch^\Delta.$ 
Then the functor $F$ factorizes  as  a composite
$$\Ch^\Delta
\stackrel{F_\S}{\longrightarrow}\Ch^\Delta\stackrel{ev_q}{\longrightarrow}\Chain.$$
Moreover, the full subcategory $\Ch$ is the category of strongly
dualizable objects in $\Chain$. We conclude that $F$ is corepresented
by Lemmas \ref{corepresentable} and \ref{endo}. 
\end{proof}

\noindent 
{\bf Analysis of
$F_\S(D[q])$.}
Let us describe  $F_\S(D[q])$
explicitly.
{\rm
To fix the notation, recall that $\Delta[q]$ is a simplicial set with
\[
\Delta[q]_n 
= \{(\Rada a0n);\ 0 \leq a_0 \leq \cdots \leq a_n \leq q\}, \ n \geq 0,
\]
the simplicial structure being given by
\begin{eqnarray*}
\partial_i(\Rada a0n) &:=& (\Rada a0{i-1},\Rada a{i+1}n), \mbox { and}
\\
s_i(\Rada a0n) &:=& (\Rada a0{j},\Rada a{j}n), \ \mbox { for }
0 \leq i \leq n.
\end{eqnarray*}
The simplicial abelian group $M[q]$ generated by $\Delta[q]$ has
the basis 
\[
\{(\Rada a0n);\ 0 \leq a_0 \leq \cdots \leq a_n \leq q\}.
\]
Let us denote by 
\[
\{\langle \Rada a0n\rangle;\ 0 \leq a_0 \leq \cdots \leq a_n \leq q\}.
\]
the dual basis of $D[q]^n$. The coboundary
operators are then given by
\[
d_i \langle \Rada a0n \rangle = \sum_{a_{i-1} \leq s \leq a_i}
\langle \Rada a0{i-1},s,\Rada a{i+1}n \rangle
\]
where we assume $a_{-1} := 0$ and $a_{n+1} := q$ so the formula makes
sense also for $i = 0$ or $i = n+1$.
For instance, for the generator $\langle 1,1 \rangle \in D[3]^1$ we have
\[
d_0 \langle 1,1 \rangle = \langle 0,1,1 \rangle + \langle 1, 1,1
\rangle,\
d_1 \langle 1,1 \rangle =  \langle 1,1,1 \rangle,\ 
d_2 \langle 1,1 \rangle = \langle 1,1,1 \rangle + \langle 1, 1,2
\rangle + \langle 1,1,3 \rangle.
\]
The above formulas show that $D[q]$ is not spanned by a cosimplicial set.

It is clear from the above explanation that elements of $F_\S(D[q])^n$
are linear combinations of the expressions of the form
\[
(\langle \Rada a0n \rangle,\sigma),\ \mbox { for } 
0 \leq a_0 \leq \cdots \leq a_n \leq q,\
\sigma \in S_n.
\]
The differential in $\bs(F_\S(D[q]))$ is given by the formula
\[
d^n (\langle \Rada a0n \rangle,\sigma) = 
\sum_{0 \leq i \leq n+1} (-1)^i (d_{\bsigma(i)}\langle \Rada a0n
\rangle,d_i(\sigma) ),
\]
where the definition of $\bsigma(i)$ and $d_i(\sigma)$ was recalled in
the proof of Proposition~\ref{p44}.
}

In~\cite[page~481]{markl:de} we introduced simple permutations whose
definition we briefly recall. We defined first, for each $\sigma \in
S_n$, a natural number $g(\sigma)$, $0 \leq g(\sigma) \leq n$,
which we called the {\em grade\/} of $\sigma$, as follows.  The grade
of the unit $\unit_n \in S_n$ is $n-1$, $g(\unit_n) := n-1$.  For
$\sigma \not= \unit_n$, let
\begin{eqnarray*}
a(\sigma) &:=& \max\{i;\ \sigma = \unit_i \times \tau \mbox { for some }
\tau \in S_{n-i}\}, \mbox { and }
\\
c(\sigma) &:=& \max\{j;\ \sigma = \lambda \times \unit_j \mbox { for some }
\lambda \in S_{n-j}\}.
\end{eqnarray*}
There clearly exists a unique $\omega = \omega(\sigma) \in S_{n
  - a(\sigma) - c(\sigma)}$ such that
$\sigma = \unit_{a(\sigma)} \times \omega(\sigma) \times
\unit_{c(\sigma)}$.
Let, finally, $b(\sigma)$ be the number of ``doubled strings'' in
$\omega(\sigma)$, 
\[
b(\sigma) := \{1 \leq s < k;\ \omega(s+1) = \omega(s) +1\}.
\]
The grade of $\sigma$ was then defined by 
\[
g(\sigma) := a(\sigma) + b(\sigma) + c(\sigma),
\]
see Figure~1 of~\cite{markl:de} for examples.
We called $\sigma \in S_k$, $k \geq 1$, {\em simple\/} if
$g(\sigma) = 0$. 

Now, for each $\sigma \in S_n$, $\sigma \not= \unit_n$, 
we define a unique simple $\kappa
= \kappa(\sigma) \in S_k$, $k = n - g(\sigma)$, by contracting all
``multiple strings'' of $\omega(\sigma)$ into ``simple'' ones. 
We put $\kappa(\unit_n):= \unit_1$. Since, for $\sigma \in S_n$ and $0 \leq i \leq
n + 1$, $\kappa(\sigma) = \kappa(d_i (\sigma))$, for each cosimplicial
chain complex $C^\bullet$ one has the decomposition
\begin{equation}
\label{eq98}
F_\S(C^\bullet) = 
\bigoplus_{\chi \mbox {\scriptsize \rm\ simple}} F^\chi_\S(C^\bullet)
\end{equation}
in which the subspaces
\[
F^\chi_\S(C^\bullet) := \{(a,\sigma)\in F_\S(C^\bullet); 
\ \kappa(\sigma) = \chi \}
\]
are invariant under coboundary operators. Let us see what happens if
$C^\bullet = D[q]$.

Consider a simple permutation  $\chi \in S_m$ different from $\unit_1 \in
S_1$. We claim that there is a direct sum decomposition
\begin{equation}
\label{eq556}
F^\chi_\S(D[q]) = \bigoplus_{0 \leq x_0 \leq \cdots \leq x_m \leq q}
F^\chi_\S(D[q])_{\Rada x0m},
\end{equation}
where $F^\chi_\S(D[q])_{\Rada x0m}$ is the smallest subspace of
$F_\S(D[q])$ invariant under the coboundary operators and
containing 
\[
(\langle \Rada x0m \rangle , \chi) \in F^\chi_\S(D[q])^m.
\] 
Before we formulate a formal argument that $F^\chi_\S(D[q])$ indeed
decomposes as in~(\ref{eq556}), we give the following
example.

\begin{example}
{\rm
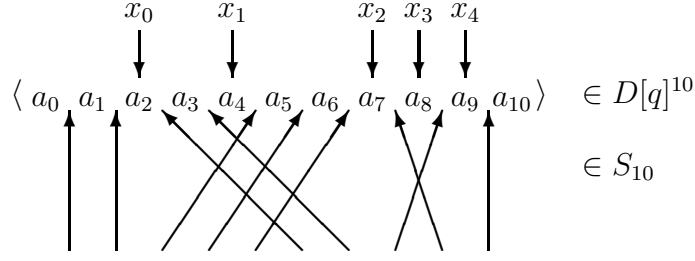
\begin{figure}[t]
\[
\setlength{\unitlength}{1,5em}
\begin{picture}(12,5)
\thicklines
\put(0,0){\vector(0,1){3}}
\put(1,0){\vector(0,1){3}}
\put(2,0){\vector(2,3){2}}
\put(3,0){\vector(2,3){2}}
\put(4,0){\vector(2,3){2}}
\put(5,0){\vector(-1,1){3}}
\put(6,0){\vector(-1,1){3}}
\put(7,0){\vector(1,3){1}}
\put(8,0){\vector(-1,3){1}}
\put(9,0){\vector(0,1){3}}
\put(-1,3){\makebox(0,0)[br]{$\langle$}}
\put(-.5,3){\makebox(0,0)[b]{$a_0$}}
\put(.5,3){\makebox(0,0)[b]{$a_1$}}
\put(1.5,3){\makebox(0,0)[b]{$a_2$}}
\put(2.5,3){\makebox(0,0)[b]{$a_3$}}
\put(3.5,3){\makebox(0,0)[b]{$a_4$}}
\put(4.5,3){\makebox(0,0)[b]{$a_5$}}
\put(5.5,3){\makebox(0,0)[b]{$a_6$}}
\put(6.5,3){\makebox(0,0)[b]{$a_7$}}
\put(7.5,3){\makebox(0,0)[b]{$a_8$}}
\put(8.5,3){\makebox(0,0)[b]{$a_9$}}
\put(9.5,3){\makebox(0,0)[b]{$a_{10}$}}
\put(10,3){\makebox(0,0)[bl]{$\rangle$}}
\put(1.5,4.7){\vector(0,-1){1}}
\put(3.5,4.7){\vector(0,-1){1}}
\put(6.5,4.7){\vector(0,-1){1}}
\put(7.5,4.7){\vector(0,-1){1}}
\put(8.5,4.7){\vector(0,-1){1}}
\put(1.5,4.9){\makebox(0,0)[b]{$x_0$}}
\put(3.5,4.9){\makebox(0,0)[b]{$x_1$}}
\put(6.5,4.9){\makebox(0,0)[b]{$x_2$}}
\put(7.5,4.9){\makebox(0,0)[b]{$x_3$}}
\put(8.5,4.9){\makebox(0,0)[b]{$x_4$}}
\put(11,3){\makebox(0,0)[lb]{$\in D[q]^{10}$}}
\put(11,1.5){\makebox(0,0)[lb]{$\in S_{10}$}}
\end{picture}
\]
\caption{\label{fig1} An element $(a,\sigma) \in F_\S(D[q])^{10}$.}
\end{figure}
Figure~\ref{fig1} symbolizes an element $(a,\sigma) \in
F_\S(D[q])^{10}$, with $a = \langle \Rada a0{10} \rangle$ and $\sigma
\in S_{10}$ the permutation determined by the arrows at the bottom
of the picture. The corresponding simple permutation $\chi :=
\kappa(\sigma)$ is
\[
\setlength{\unitlength}{1,5em}
\begin{picture}(3,1)
\thicklines
\put(0,0){\vector(1,1){1}}
\put(1,0){\vector(-1,1){1}}
\put(2,0){\vector(1,1){1}}
\put(3,0){\vector(-1,1){1}}
\put(4,.5){\makebox(0,0)[l]{$\in S_4$}}
\end{picture}
\]
and $(a,\sigma)$ belongs to the subspace $F^\chi_\S(D[q])_{\Rada x04}$
generated by
\[
\setlength{\unitlength}{1,5em}
\begin{picture}(3,2)
\thicklines
\put(0,0){\vector(1,1){1}}
\put(1,0){\vector(-1,1){1}}
\put(2,0){\vector(1,1){1}}
\put(3,0){\vector(-1,1){1}}
\put(-1,1){\makebox(0,0)[br]{$\langle$}}
\put(-.5,1){\makebox(0,0)[b]{$x_0$}}
\put(.5,1){\makebox(0,0)[b]{$x_1$}}
\put(1.5,1){\makebox(0,0)[b]{$x_2$}}
\put(2.5,1){\makebox(0,0)[b]{$x_3$}}
\put(3.5,1){\makebox(0,0)[b]{$x_4$}}
\put(4,1){\makebox(0,0)[bl]{$\rangle$}}
\put(4.6,.8){\makebox(0,0)[l]{$\in F_\S(D[q])^4$,}}
\end{picture}
\]
where $(x_0,x_1,x_2,x_3,x_4) := (a_2,a_4,a_7,a_8,a_9)$.
}
\end{example}

Let us describe how to decide to which subspace
$F^\chi_\S(D[q])_{\Rada x0m}$ a general $(a,\sigma) \in F_\S(D[q])$
belongs. The permutation $\chi$ is determined easily, it is the simple
permutation $\chi := \kappa(\sigma) \in S_m$ associated to $\sigma$. If
$a = \langle \Rada a0n \rangle$, then $\Rada x0m$ is the sequence
obtained from $\Rada a0n$ by deleting $a_i's$ located between doubled
arrows of $\sigma$, between $\langle$ and a vertical arrow $\uparrow$, 
or between $\uparrow$ and $\rangle$. We believe that everything is
clear from Figure~\ref{fig1}. Let finally $a_\sigma := \langle \Rada x0m
\rangle$. It is easy to verify that then 
\[
F^\chi_\S(D[q])_{\Rada x0m} = \{(a,\sigma) \in  F_\S(D[q]); \
\kappa(\sigma) = \chi,\ a_\sigma = \langle \Rada x0m \rangle \}
\]
which makes~(\ref{eq556}) obvious.

Let $\whD[q]$ be the augmented cosimplicial group obtained by adding
to $D[q]$ the piece $\whD[q]^{-1}$ spanned by the empty bracket
$\langle \ \rangle$ and $d_0: \whD[q]^{-1} \to \whD[q]^{0}$ defined by
\[
d_0 \langle \ \rangle := \sum_{0 \leq s \leq q} \langle s \rangle.
\]
We claim that, for each simple permutation $\chi \in S_m$ different
from $\unit \in S_1$, there is an isomorphisms of cochain complexes
\begin{equation}
\label{Mulhouse}
\bs(F^\chi_\S(D[q])_{\Rada x0m})
\cong\ \uparrow^{2m+2}
\bs(\whD[x_0]) \otimes  \bs(\whD[x_1\! -\! x_0]) \otimes \cdots \otimes
\bs(\whD[x_m\!  -\! x_{m-1}])   \otimes  \bs(\whD[q\! -\! x_m ]),
\end{equation} 
where $\uparrow^{2m+2}$ in the right hand side means that 
the degrees are shifted up by $2m+2$. For instance, to $(a,\sigma) \in
F_\S(D[q])^{10}$ in
Figure~\ref{fig1} it corresponds
\begin{eqnarray*}
\lefteqn{
\langle  a_0,a_1  \rangle \otimes 
\langle  a_3  \rangle \otimes \langle  a_5,a_6  \rangle \otimes 
\langle  \  \rangle \otimes \langle \ \rangle \otimes \langle  a_{10}  \rangle
\in}
\\
&& \whD[a_2]^1 \otimes \whD[a_4 - a_2]^0 \otimes \whD[a_7 - a_4]^1
\otimes \whD[a_8 - a_7]^{-1}  \otimes \whD[a_9 - a_8]^{-1} 
\otimes \whD[q - a_9]^{0}. 
\end{eqnarray*}

Isomorphism~(\ref{Mulhouse}) is defined as follows: to each
$(\langle \Rada a0n \rangle,\sigma) \in F^\chi_\S(D[q])_{\Rada x0m}^n$
one associates the element of the tensor product in the right hand
side obtained by replacing,  in $\langle \Rada a0n \rangle$, 
each $x_j \in \{\Rada x0m\}$ by $\rangle \otimes \langle$.

Combining isomorphisms~(\ref{eq98}),~(\ref{eq556}) and~(\ref{Mulhouse}) 
with the obvious equality 
$F^{\unit}_\S(D[q]) \cong D[q]$, we~finally~get:

\begin{lemma}
\label{l1}
For each $q \geq 1$, one has the isomorphism of cochain
complexes: 
\begin{eqnarray*}
\lefteqn{
\bs(F_\S(D[q])) \cong}
\\
&& \cong \bs(D[q]) \oplus 
\bigoplus
 \uparrow^{2m+2}
\bs(\whD[x_0]) \otimes  \bs(\whD[x_1\! -\! x_0]) \otimes \cdots \otimes
\bs(\whD[x_m\!  -\! x_{m-1}])   \otimes  \bs(\whD[q\! -\! x_m ]),
\end{eqnarray*}
where the summation in the right hand side is taken over all simple
permutations $\chi \in S_m$ different from $\unit \in
S_1$ and all $0 \leq x_0 \leq \cdots \leq x_m \leq q$. 
\end{lemma}

\begin{proposition}
\label{p22}
For each $q \geq 0$, $H^{\geq 1}(\bs(F_\S(D[q]))) = 0$. 
\end{proposition}
\begin{proof} It is a
standard fact that $H^i (\bs(D[q])) = 0$ for all $i \geq 1$ and that $H^i
(\bs(\whD[q])) = 0$  for all $i$. The proposition then follows from 
Lemma~\ref{l1} and the K\"unneth theorem.
\end{proof}

\section{Bar construction and proof of Theorem~C}
\label{JARka}

In this section we analyze the bar construction of an associative
algebra and the induced structures on the Hochschild complex. We then
formulate Proposition~\ref{JarkA} relating the operads $\Big$ and
$\Tam$, and prove Theorem~C.

As everywhere in this paper, we work over the ring $\bbZ$ of integers,
though all results and definitions in this section remain valid over
an arbitrary commutative unital ground ring $R$, only ``abelian
group'' must be translated into ``$R$-module.''  Let $\bar* =
B_*(A,A,A)$ be the two-sided bar construction of a unital associative
algebra~$A = (A,\mu,1)$ so that $\bar n = A \ot \otexp An \ot A$ is
the free $A$-bimodule on $\otexp An$, $n \geq 0$. The bar construction
extends into a functor $\bar * : \I S \to \AbiMod$ from the crossed
interval group $\I S$ introduced in Definition~\ref{IS} to the
category $\AbiMod$ of $A$-bimodules as follows.

On objects, $\bar *(\int n) := \bar {n+1}$.  For $a_{-1} \ot a_0 \ot
\cdots \ot a_m \ot a_{m+1} \in \bar {m+1}$ and a~morphism $f : \int m
\to \int n$ in $\IS$ put
\[
\bar *(f)(a_{-1} \ot a_0 \ot \cdots \ot a_m \ot a_{m+1}) :=
\overline a_{-1} \ot \overline a_0 \ot \cdots \ot \overline a_n \ot
\overline a_{n+1} \in \bar {n+1},
\]
where $\overline a_i$ is, for $-1 \leq i \leq n+1$, the product of
$a_j$'s, $j \in f^{-1}(i)$, in the order specified by the linear order
on the fiber  $f^{-1}(i)$. 
If $f^{-1}(i) = \emptyset$ we put $\overline a_i := 1$.

It is easy to see that the composition $\Delta^\op
\stackrel{\joy}{\cong} \I \hookrightarrow \I S \stackrel
{\bar*}\longrightarrow \AbiMod$ is the standard simplicial
$A$-$A$-bimodular structure of the bar construction. The functor $\bar
* :\I S \to \AbiMod $ is in fact induced from the {\em symmetric bar
construction\/} $B^{\it sym}_*(A) : \Delta S \to \Ab$ introduced
in~\cite{fiedorowicz}, by the~diagram

\[
\Square 
{\IS}{\AbiMod}{\Delta S}{\Ab}%
{\bar*}\iota{\it forget}{B^{\it sym}_*(A)}
\]   
in which $\iota$ is the inclusion~(\ref{Napise_Magda_jeste?}) and
$\Ab$ the category of abelian groups. 

Consequently, for any $A$-bimodule $M$, the space $C^*(A;M) :=
\Hom_{\mbox{\scriptsize $A$-$A$}}(\bar *,M) \cong \Lin(\otexp A*,M)$
of Hochschild cochains of $A$ with coefficients in $M$ is an $(\I
S)^\op$-abelian group.  This structure will be crucial for us.

Let us turn our attention to the $\bbN$-coloured operad $B$ from
Section~\ref{Jarce_umrela_maminka.}. As we already noted, the spaces
$B^l_k$ (i.e.~$B^l_{\Rada k1n}$ with $n=1$ and $k = k_1$) are the
$\Hom$-sets of the underlying category $\calU(B)$ of this coloured operad.  It
turns out that this category is in fact a linearization of $(\I
S)^{\op}$, by which we mean that
\[
\Span(\IS ( \int {l-1},\int{k-1}) \stackrel{\cong}\longrightarrow B^l_k
\]
for each $k,l \geq 0$, where $\Span(-)$ denotes the free abelian group
functor.  

To describe the above isomorphism, it will be convenient to interpret
elements of $B$ as operations on the Hochschild cochain complex of
$A$, so we assume that $A$ is {\em generic\/}.  The isomorphism is
then induced by the set-map that assigns to each morphism $g : \int
{l-1} \to \int{k-1}$ in $\I S$ the natural operation $O_g : \CH k \to
\CH l \in B^l_k$ constructed as follows.

Let $\barg : \{\rada 0{l-1}\} \to \{\rada {-1}{k}\} = \int {k-1}$ be the
restriction $\barg := g|_{\{\rada 0{l-1}\}}$ and denote, for each $-1 \leq
i \leq k$ for which $\barg^{-1} (i)$ is non-empty
\[
\barg^{-1} (i) = (\rada{\sigma^i_1}{\sigma^i_{s(i)}}),
\]
where the numbers $\rada{\sigma^i_1}{\sigma^i_{s(i)}} \in
\{\rada 0{l-1}\}$ are ordered according to the restriction of the
order of the fiber $g^{-1} (i)$ to its subset
$\barg^{-1}(i) \subset  g^{-1} (i)$.

For $f: \otexp Ak \to A\in \CH k$ the cochain $O_g(f) : \otexp Al
\to A\in \CH l$ is given by the formula
\[
O_g(f)(\Rada a0{l-1}) = \bara_{-1}\cdot f(\Rada \bara0{k-1})\cdot \bara_k,
\]
where, for $\Rada a0{l-1} \in A$ and $i \in \{-1,\ldots,k\}$,
\[
\bara_i := \cases {a_{\sigma^i_1} \cdot \ldots \cdot a_{\sigma^i_{s(i)}}}{if
      $\barg^{-1} (i) \not=\emptyset$, and}{1}{otherwise.}
\]
The $(l;k)$-tree encoding the operation $O_g$ is depicted in
Figure~\ref{fig44}. 

\begin{figure}[t]
{
\thicklines
\unitlength=1.0pt
\begin{picture}(160.00,150.00)(0.00,0.00)
\put(147.00,50.00){\makebox(0.00,0.00){$\barg^{-1}(k)$}}
\put(130.00,0.00){\makebox(0.00,0.00){$\barg^{-1}(k\!-\!1)$}}
\put(35.00,0.00){\makebox(0.00,0.00){$\barg^{-1}(0)$}}
\put(15.00,50.00){\makebox(0.00,0.00){$\barg^{-1}(-1)$}}
\put(70.00,10.00){\makebox(0.00,0.00){$\cdots$}}
\put(125.00,10.00){\makebox(0.00,0.00){$\cdots$}}
\put(35.00,10.00){\makebox(0.00,0.00){$\cdots$}}
\put(145.00,60.00){\makebox(0.00,0.00){$\cdots$}}
\put(17.00,60.00){\makebox(0.00,0.00){$\cdots$}}
\put(100.00,60.00){\line(4,-5){40.00}}
\put(100.00,60.00){\line(-2,3){19.00}}
\put(100.00,60.00){\line(1,-5){10.00}}
\put(60.00,60.00){\line(-1,-5){10.00}}
\put(60.00,60.00){\line(-4,-5){40.00}}
\put(60.00,60.00){\line(2,3){19.00}}
\put(80.00,120.00){\line(-4,-3){80.00}}
\put(80.00,120.00){\line(-5,-6){50.00}}
\put(80.00,120.00){\line(5,-6){50.00}}
\put(80.00,120.00){\line(4,-3){80.00}}
\put(85.00,58.00){\makebox(0.00,0.00){$\cdots$}}
\put(68.00,52.00){\makebox(0.00,0.00)[t]{\scriptsize$1$}}
\put(112.00,60.00){\makebox(0.00,0.00)[b]{\scriptsize$k\!-\!1$}}
\put(53.00,60.00){\makebox(0.00,0.00)[b]{\scriptsize$0$}}
\put(75.00,120.00){\makebox(0.00,0.00)[br]{\scriptsize$(-1,k)$}}
\put(100.00,60.00){\makebox(0.00,0.00){$\bullet$}}
\put(60.00,60.00){\makebox(0.00,0.00){$\bullet$}}
\put(80.00,120.00){\makebox(0.00,0.00){$\bullet$}}
\put(70.00,58.00){\makebox(0.00,0.00){$\bullet$}}
\put(80.00,90.00){\makebox(0.00,0.00){\large$\circ$}}
\put(80.00,88.00){\line(-1,-3){10}}
\put(80.00,150.00){\line(0,-1){57.00}}
\end{picture}}
\caption{\label{fig44} 
The $(l;k)$-tree representing the operation $O_g \in B^l_k$. It has
one white vertex and $k+1$ black vertices. To each black vertex one
attaches legs as indicated in the picture, labeled by the elements of
the fibers of $\barg$. If the corresponding fiber is empty, the black
vertex becomes a special one -- in the above picture this is 
the vertex labeled $1$.  
}
\end{figure}
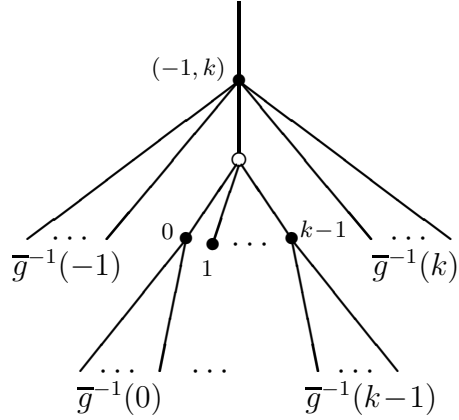

It follows from general properties of coloured operads that $B^l_k$
acts on $B^l_{\Rada k1n}$ covariantly on the upper index and
contravariantly on each lower index.  Therefore, the spaces
$B^l_{\Rada k1n}$ assemble into a functor
\[
\SB^{\bullet}_{\bullet_1,\ldots,\bullet_n} : (\I S)^{\op} \times
(\I S)^n \to \Ab.
\]
Precomposing this functor with the canonical functor 
\[
\Delta\times (\Delta^{\op})^n \rightarrow (\I S)^{\op} \times
(\I S)^n
\] 
we get the functor $B^{\bullet}_{\bullet_1,\ldots,\bullet_n}$
in~(\ref{eq:2}).

\begin{proposition} 
\label{JarkA}
The functor $\SB^{\bullet}_{\bullet_1,\ldots,\bullet_n} : (\I S)^{\op}
\times (\I S)^n \to \Ab$ is the result of application of the functor
$F_{\S}$ described in Lemma~\ref{freeHobject} to the functor
$T^{\bullet}_{\bullet_1,\ldots,\bullet_n}.$ Therefore,
\[
\Big = \uTot(\lTot (F_{\S}T^{\bullet}_{\bullet_1,\ldots,\bullet_n} ))
= \uTot( F_{\S}\lTot(T^{\bullet}_{\bullet_1,\ldots,\bullet_n}) ),
\] 
and the inclusion $\Tam \rightarrow \Big$ is induced by the unit of
the monad $F_{\S} .$ The same relationship holds between $\TN$ and
$\BN$.
\end{proposition}

\begin{proof}
The proposition easily follows from the explicit description of $\Big$ and
$\Tam$ given on pages~\pageref{zacatek}--\pageref{konec} of 
Section \ref{Jarce_umrela_maminka.}.
\end{proof}

We can finally give

\begin{proof}[Proof of Theorem~C]
The inclusion $\Tam \hookrightarrow \Big$ is a weak equivalence by
Theorem~\ref{t2} and Proposition~\ref{JarkA}. It is shown in the proof 
of~\cite[Theorem~5.9]{bbm} that the inclusion $\Br \hookrightarrow \Tam$ is a 
weak equivalence, too.  All vertical
projections $\Big\epi \Norm(\Big)$, $\Tam\epi \NormT$ and $\Br \epi
\Norm(\Br)$ are weak equivalences because they are normalization maps.
So, all morphisms in the two upper triangles and between them are weak
equivalences by two out of three property, which gives~(2).

To establish the homotopy type of the operads in the upper two
triangles of Figure~\ref{fig3}, it is, by~(2), enough to establish the
homotopy type of one of them. 

The fact that $\Tam$ has the
homotopy type of $\SC_{-*}(\Dis)$ follows from
~\cite{BataninBerger}. It was shown, 
in the proof of Proposition~2.14 of that paper, that the operad $T$
(denoted ${\mathcal O}$ there) is isomorphic to the second filtration
$\Lat 2$ of the lattice path operad $\Latnic$. Theorem~3.8
of~\cite{BataninBerger} then claims that the totalization of $\Lat 2$,
isomorphic to $\Tam$ due to the isomorphism $T \cong \Lat 2$, is
an $E_2$-operad in $\Chain$.

Item (4) of the theorem is obvious from the analysis 
in Subsection~\ref{Jarunka} which immediately implies 
that the inclusion $\BN \hookrightarrow \Big$
induces homology isomorphisms $\BN(n) \sim \Big(n)$ for each $n \geq
1$. The same relationship holds also for the suboperads $\TN \subset
\Tam$ and $\widehat{\Br} \subset \Br$. 

Regarding the homotopy type of the operads in the bottom triangle of
Figure~\ref{fig3}, it suffices, given~(4), to prove that
the dg-abelian groups $\BN(0)$ and $\TN(0)$ are acyclic, because
$\widehat{\Br}(0)=0$ by definition.
This was done in~\cite[Example~12]{markl:de}.
\end{proof}

\appendix
\section{Acyclic Models Theorem for cochain functors.}
\label{sec:acycl-models-theor}

We introduce a theory of acyclic models for functors with values in
cochain complexes ({\it cochain functors}).  Due to the lack of a
suitable concept of cofree modules, this theory is not merely 
a~dualization of the chain version.

We present the necessary notions and results  in a 
general setting with an arbitrary complete,  cocomplete closed
symmetric monoidal category $V$ as the underlying 
category, although $V$ will,
for the purposes of this paper, 
always be the category of chain complexes $\Chain$. We choose this
more categorical approach because, in our opinion, it makes
our definitions clearer in the special case $V=\Chain$. 

We refer to the book~\cite{Kelly} for definitions of a category,
functor and natural transformation enriched in a closed symmetric
monoidal category, as well as for a discussion concerning the
difference between enriched and nonenriched natural transformations.
In particular, we remind that there can exist natural transformations
between $V$-functors, which are not $V$-transformations.

Let $(V,\otimes,I)$ be a complete and cocomplete closed symmetric monoidal  category.  
Let $\D$ be a $V$-category. Let $F : \D \to V$ be a  $V$-functor and
$\frakM$ a set of objects of $\D$. Let us denote, for $X \in \D$, by
$\hatF(X)$ the following object of $V$:
\begin{equation}
\label{e1}
\hatF(X):= \underline{V}(\coprod_{M \in \frakM}\uD(X,M),F(M)) 
= \prod_{M \in \frakM} \underline{V}(\uD(X,M),F(M)).
\end{equation}
In the above display, as well as in the rest of this section, 
we used the notation $\underline{\C}(-,-)$ 
for the  enriched $\Hom$-functor of a $V$-category $\C$.

It is obvious that~(\ref{e1}) defines a $V$-functor $\hatF : \D \to
V.$  
There exists the canonical  $V$-enriched transformation $\lambda : F
\to \hatF$ such that, for $X\in \D$ and $M\in \frakM$, 
 the component of $\lambda$  
\[
\lambda(X): F(X) \to \underline{V}(\uD(X,M),F(M))
\]
is the morphism in $V$ which corresponds under adjunction to the evaluation morphism
$$F(K)\otimes \uD(X,M) \stackrel{1\otimes F}{\longrightarrow} F(X)\otimes \underline{V}(F(X),F(M)) \rightarrow F(M).$$

\begin{definition}
\label{d1}
We say that $F: \D \to V$ as above is {\em corepresented with
models\/} $\frakM$ if there exists a (not necessarily $V$-enriched)
transformation $\chi : \hatF \to F$ such that $\chi \lambda = \id_F$.
\end{definition}

Recall~\cite[page~173]{may:equivariant} that an object
$K\in V$ is called {\it strongly dualizable} if the morphism
\[
\nu:K^{\#}\otimes K \rightarrow \underline{V}(K,K)
\] 
is an isomorphism. Here, $K^{\#}: = \underline{V}(K,I)$ and $\nu$ is adjoint
to the morphism 
\[
\underline{V}(K,I)\rightarrow \underline{V}(K\otimes
K, I\otimes K)\simeq \underline{V}(K\otimes K, K).
\] 
Let $\C$ be a small $V$-enriched category.  Assume that, for any two
$a,b\in \C$, the object $\underline{\C}(a,b) \in V$ is strongly
dualizable. Let $C_q, q\in \C$, be a representable presheaf i.e.~the
functor $C_q: \C^{\op}\rightarrow V$ given by $C_q(b) :=
\underline{\C}(b,q)$, for $b \in \C$. Let, finally, $D_q : \C
\rightarrow V$ be the functor defined as $D_q (b) := C_q^{\#}(b)$.

Let $\Dual\subset V$ be the full subcategory of strongly
dualizable objects and
$\Dual^\C$ the enriched category of enriched functors $\C \to \Dual$ and
their enriched natural transformations. Let $q$ be an object of~$\C$.  
An ample source of
examples of corepresented functors is provided by the following lemma.
\begin{lemma}
\label{corepresentable} Let $q\in \C.$
The enriched functor $ev_q: \Dual^\C \rightarrow V$, $ ev_q(X) := X(q)$, is
corepresented  with the model $D_q .$ 
\end{lemma}

\begin{proof} 
Observe that dualization functor $(-)^{\#}$  induces a contravariant isomorphism of categories
$\Dual^\C \cong \Dual^{\C^{op}}.$ Then 
\begin{eqnarray*} \widetilde{ev_q}(X) & = & \underline{V}(\underline{\Dual^\C}(X,D_q),D_q(q))   
\cong \underline{V}(\underline{\Dual^{\C^{op}}}(D^{\#}_q,X^{\#}),\underline{\C}(q,q)^{\#})
\\
\nonumber &\cong& \underline{V}(\underline{\Dual^{\C^{op}}}(C_q,X^{\#}),\underline{\C}(q,q)^{\#})
\cong \underline{V}(X^{\#}(q),\underline{\C}(q,q)^{\#})
\\
\nonumber
&\cong& \underline{V}(\underline{\C}(q,q),X(q)).\end{eqnarray*}
Then we define $\chi:\widetilde{ev_q}(X)\to ev_q(X)$ by applying $\underline{V}(-,X(q))$ to the unit morphism
$I\to \underline{\C}(q,q)$ 
$$\underline{V}(\underline{\C}(q,q),X(q))\to \underline{V}(I,X(q))\cong X(q)= ev_q(X).$$
The equation  $\chi \lambda = \id_{ev_q}$ is obvious.
\end{proof}

The following lemma together with Lemma \ref{corepresentable} allows
to construct new corepresented functors.

\begin{lemma}
\label{endo} 
Let $F: \D \to V$ be corepresented with
models $\frakM$ and let $A:\D\to\D$ be an endofunctor. Then the
composite $F\circ A: \D\to V$ is corepresented with models
$\frakM.$
\end{lemma}

\begin{proof}
By direct verification we establish the equality
$\widetilde{(F\circ A)}(X) = \widetilde{F}(A(X)).$  So we have that
the composite
\[
F\circ A(X)\to \widetilde{(F\circ A)}(X) \to F\circ A(X)
\]
equals the identity.
\end{proof}

For an abelian closed symmetric monoidal category $V$ the category
$\Coch V$ of cochain complexes in $V$ has a canonical $V$-enrichment.
The following theorem is a dual version of the classical {\em Acyclic
Models Theorem\/} (see, for instance,~\cite[Theorem~28.3]{may:67}).  With
our definitions in place, the arguments are standard so we leave the
proof of this theorem as an exercise.

\begin{theorem}
\label{p2}
Let $V$ be an abelian  symmetric monoidal  closed category.
Let $\D$ be a V-category and
$\Coch V$ the category of cochain
complexes in $V$. Let $A^*,B^*
: \D \to \Coch V$ be $V$-functors
and $f^*,g^* : A^* \to B^*$ natural
transformations such that $f^0 = g^0$.
Suppose that

(i) $B^q: \D \to V$ 
is corepresented with models $\frakM$, for each $q \geq 1$,
   
(ii) for $m \geq 1$ and  each $M \in \frakM$ the short exact sequence  
\[
0 \to \Ker(d^m) \to A^m(M)
\stackrel{d^m}{\to} \Im(d^m) \to 0
\]
splits, and

(iii) $H^{\geq 1}(A^*(M),d) = 0$, for each $M \in \frakM$.

\noindent 
Then there exists a natural cochain homotopy $H^* : f^* \sim
g^*$. 
\end{theorem}

When $V=\Chain$, Theorem \ref{p2} implies:

\begin{proposition}
\label{p1}
Let $\D$ be a dg-category and $\CochChain$ the category of cochain
complexes of chain complexes. Let $A^*_*,B^*_* : \D \to \CochChain$ be
dg-functors and $f^*_*,g^*_* : A^*_* \to B^*_*$ (not necessarily dg-enriched) 
transformations such that $f_*^0 = g_*^0$.  Suppose moreover that

(i) $B_*^q: \D \to \Chain$ 
is corepresented with models $\frakM$, for each $q \geq 1$,
   
(ii) $A_*^q(M)$ is a finitely generated torsion-free
chain complex concentrated in degree $0$, 
for each $q \geq 1$, $M \in \frakM$, and

(iii) $H^{\geq 1}(A_0^*(M),d) = 0$, for each $M \in \frakM$.

\noindent 
Then there exists a natural cochain homotopy $H^*_* : f^*_* \sim
g^*_*$ that commutes with the vertical (chain) differentials and  
determines, for each $K \in \D$, a natural cochain homotopy between
the induced maps
$|f^*_*|^*, |g^*_*|^*: |A^*_*(K)|^* \to |B^*_*(K)|^*$.
\end{proposition}

\begin{proof}
The only issue that has to be verified is that (ii) of
Proposition~\ref{p1} implies (ii) of Theorem~\ref{p2}, which is
simple. The rest is obvious.
\end{proof}

\section{Generic algebras}
\label{Jarunka_jeste_spinka}

This section is devoted solely to our proof of the following statement.

\begin{theorem}
\label{snad_mne_neopusti}
The free associative unital algebra $U := \bbT (x_1,x_2,x_3,\ldots)$ generated
by countably many generators $x_1,x_2,x_3,\ldots$ is generic.
\end{theorem}

The theorem will follow from Lemma~\ref{l9} below.
The underlying abelian group of $U$ is free, 
with a preferred basis given by monomials. 
{}For $e\in U$ and a monomial $h$ we denote by $e|_h\in \bbZ$ the coefficient 
at $h$ in this preferred basis. 

\begin{lemma}
\label{l9}
For each $(l;\Rada k1n)$-tree $T$ there exist cochains $f^T_i \in
C^{k_i}(U,U)$, $1\leq i \leq n$, 
and a monomial $h^T \in U$ such that, for any $(l;\Rada
k1n)$-tree $S$  
\[
O_S(f^T_1,\ldots,f^T_n)(\Rada x1l)|_{h^T}  =
\cases1{if $S=T$}0{otherwise.}
\]
\end{lemma}

\begin{proof}
The cochains $f^T_i : \otexp U{k_i} \to U$, $1 \leq i \leq
n$, will be determined by a
choice of monomials $u^i_1,\ldots,u^i_{k_i}\in U$ by requiring that
\begin{equation}
\label{Miluje-mne?}
f^T_i(u^i_1,\ldots,u^i_{k_i}) := x_{l+i},\ 1 \leq i \leq n,
\end{equation}
while $f^T_i(\Rada v1{k_i}) = 0$ for all monomials $\Rada v1{k_i}\in U$
such that $(\Rada v1{k_i}) \not = (u^i_1,\ldots,u^i_{k_i})$ in $\otexp
U{k_i}$. We call $x_{l+i+1}$ the {\em value\/} of $f_i$.

The $j$th member $u^i_j$, $1 \leq j \leq k_i$, of the sequence
$u^i_1,\ldots,u^i_{k_i}$ is defined as follows.
Let $m^i_j$ be the $j$th input of $f^T_i$ in
the expression $O_T(f^T_1,\ldots,f^T_n)(\Rada x1l)$.
Clearly, $m^i_j$ is either $1$ or a string made of some $x_s$, $1 \leq
s \leq l$ and/or of some $f_t^T(-,\ldots,-)$, $1 \leq t \leq n$. Then
$u^i_j$ is given by 
replacing all
function symbols in $m^i_j$ by their values, i.e.~by replacing
$f_t^T(-,\ldots,-)$ by $x_{l+t}$. The monomial $h^T$ is the result of
the same procedure applied to $O_T(f^T_1,\ldots,f^T_n)(\Rada x1l)$.

It is obvious that the cochains $\rada{f^T_1}{f^T_n}$ and the monomial
$h^T$ have the required property, because the tree $T$ can be uniquely
reconstructed from them. This finishes the proof of the lemma.

Let us illustrate the above construction on the $(8;3,3,1,3)$-tree in
Figure~\ref{fig2}. Since $l=8$, the values of $f^T_1,f^T_2,f^T_3,f^T_4$ are 
$x_9,x_{10},x_{11},x_{12}$, respectively. One has
\[
O_T(f^T_1,f^T_2,f^T_3,f^T_4)(\Rada x18) =
x_3 f^T_1(f^T_2(x_5x_6,1,x_8),x_1,f^T_3(x_7))f^T_4(x_4,1,x_2).
\]  
The monomials $u^1_1,u^1_2$ and $u^1_3$ are obtained from the inputs
$m^1_1= f^T_2(x_5x_6,1,x_8)$, $m^1_2 = x_1$ and 
$m^1_3 = f^T_3(x_7)$ of $f^T_1$ by replacing all
function symbols by their values, so $(u^1_1,u^1_2,u^1_3) =
(x_{10},x_1,x_{11})$. Therefore $f^T_1$ is defined by 
$f^T_1(x_{10},x_1,x_{11}) := x_9$. Similarly, one gets
\[
f^T_2(x_5x_6,1,x_8) := x_{10},\ f^T_3(x_7) := x_{11},\ 
f^T_4(x_4,1,x_2) := x_{12}  \ \mbox { and } h^T := x_3x_9x_{12}.
\]
\end{proof}

\begin{proof}[Proof of Theorem~\ref{snad_mne_neopusti}]
Since $\omega_A$ is epi, we only need to prove that, for $l,\Rada k1n
\geq 0$ and a linear combination $\sum_{t=1}^s\alpha_t T^t$ of $(l;\Rada
k1n)$-trees, $\omega_A (\sum_{t=1}^s\alpha_tT^t)=0$ implies that
$\rada{\alpha_1}{\alpha_s}=0$. For an arbitrary $r$, $1 \leq r \leq
s$, let $f_{1}^{T^r},\ldots,f_n^{T^r}$ and $h^{T^r}$ be as in
Lemma~\ref{l9} Then
\begin{eqnarray*}
0&=& \textstyle\omega_A (\sum_{t=1}^s\alpha_tT^t)(f_1^{T^r},\ldots,f_n^{T^r})(\Rada
x1n)|_{h^{T^r}}
\\
&=&\textstyle \sum_{t=1}^s \alpha_tO_{T^t}(f_1^{T^r},\ldots,f_n^{T^r})(\Rada
x1n)|_{h^{T^r}} = \alpha_r.
\end{eqnarray*}
Therefore $\alpha_r = 0$ for each $1 \leq r \leq s$.
\end{proof}

\noindent 
{\bf Problem.}
{\rm Give a characterization of generic algebras. We leave as an
 exercise to prove that, for instance, each algebra $A$ which
  contains $U$ as a subalgebra and a direct summand is generic. It is
  also clear that $U$ modulo the ideal generated by $x_i^2$, $i \geq
  1$, is generic. Does there exists a `minimal' (in an appropriate
  sense) generic algebra?  }


\def\cprime{$'$}

\end{document}